\documentclass[reqno,12pt]{amsart}
\usepackage{amssymb}
\usepackage{color}
\usepackage[margin=1in]{geometry} 
\usepackage[utf8]{inputenc}
\usepackage{tikz}
\usepackage[all]{xy}
\usepackage[bookmarksnumbered,colorlinks]{hyperref}
\usepackage{dsfont}
\usepackage[normalem]{ulem}
\usepackage{enumerate}

\usepackage[euler-digits]{eulervm}
\usepackage{enumerate}
\def\br#1\er{{\color{red}#1}} %
\def\bb#1\eb{{#1}} 
\def\bp#1\ep{{#1}} 
\newcommand{\R}{\mathds R}
\newcommand{\N}{\mathds N}

\newcommand{\cambios}{}

\newcommand{\som}{SOM}

\title[Stationary-Complete Spacetimes and pre-Randers metrics]{Stationary-Complete Spacetimes with non-standard splittings and pre-Randers metrics}
\author[J. Herrera]{J\'onatan Herrera}
\address{Departamento de Matemáticas, Edificio Albert Einstein\hfill\break\indent Universidad de Córdoba, Campus de Rabanales,\hfill\break\indent 14071 Córdoba, Spain}
\email{jherrera@uco.es}

\author[M. A. Javaloyes]{Miguel Angel Javaloyes}
\address{Departamento de Matem\'aticas, \hfill\break\indent
Universidad de Murcia, Campus de Espinardo,\hfill\break\indent
30100 Espinardo, Murcia, Spain}
\email{majava@um.es}

\date{}

 \thanks{The first author is partially supported by the Spanish Grant MTM2016-78807-C2-2-P (MINECO and FEDER funds). The second author was partially supported  by Spanish  MINECO/FEDER project reference
 	MTM2015-65430-P and Fundaci\'on S\'eneca
 	(Regi\'on de Murcia) project 19901/GERM/15. }

\begin{document}
\newtheorem{thm}{Theorem}[section]
\newtheorem{prop}[thm]{Proposition}
\newtheorem{lemma}[thm]{Lemma}
\newtheorem{cor}[thm]{Corollary}
\newtheorem{conv}[thm]{Convention}
\theoremstyle{definition}
\newtheorem{defi}[thm]{Definition}
\newtheorem{notation}[thm]{Notation}
\newtheorem{exe}[thm]{Example}
\newtheorem{conj}[thm]{Conjecture}
\newtheorem{prob}[thm]{Problem}
\newtheorem{rem}[thm]{Remark}
\maketitle

\begin{abstract}
Using the relativistic Fermat's principle, we establish a bridge between stationary-complete manifolds which satisfy the observer-manifold condition and pre-Randers metrics, namely, Randers metrics without any restriction on the one-form. As a consequence, we give a description of the causal ladder of such spacetimes in terms of the elements associated with the pre-Randers metric: its geodesics and the associated distance. We obtain, as applications of this interplay, the description of conformal maps of Killing submersions, and existence and multiplicity results for geodesics of pre-Randers metrics and magnetic geodesics.
\end{abstract}

\tableofcontents

\section{Introduction}
 The main goal of this paper is to establish a bridge between stationay-complete spacetimes satisfying the observer manifold condition and pre-Randers metrics, which are metrics with the same expression as Randers metrics, but without any restriction on the one-form $\omega$ (see \eqref{preRanders}).  Let us recall, on the one hand, that a spacetime is stationary-complete if it admits a complete timelike Killing vector field. If, additionally, the  spacetime satisfies the observer manifold condition, then there exists a splitting $\R\times S$ analogous to the one for standard stationary spacetimes, but with the slices not necessarily spacelike. The class of the spacetimes of this type on the product manifold $\R\times S$ is denoted by $\som(\R\times S)$ and we will refer to a spacetime of this type as an \som-spacetime 
(see Definition \ref{def:classspacetimes}). 

On the other hand, pre-Randers metrics are not necessarily positive, so they are not pseudo-Finsler metrics. In fact, its fundamental tensor is degenerate in the directions where the metric is zero. In any case, a pre-Randers metric is pre-Finsler (see Definition \ref{preFinsler}) and then we can define both, the length of a curve and an associated distance. It turns out that the associated distance is continuous (see Definition \ref{def:distance} and Proposition \ref{prop:continuidad}). The interest of pre-Randers metrics arises naturally from its applications, among which we can cite the study of magnetic geodesics (see \S \ref{magnetic} and especially Proposition \ref{magranders}). A fundamental element of pre-Randers metrics is the concept of geodesic, which cannot always be defined using the Chern connection because  the fundamental tensor can be degenerate. For this reason, we define geodesics as affinely parametrized critical points of the length functional (see Definition \ref{def:distance}). 

The key result that relates \som-spacetimes $(\R\times S,g)$ with pre-Randers metrics is Proposition \ref{prop:Fermatmetric}, which is a consequence of the relativistic Fermat's principle  \cite{Perlick90}. It says that lightlike geodesics of $(\R\times S,g)$ project into pre-geodesics of a pre-Randers metric  $F$ in $S$ given by \eqref{eq:wsfinsler}, called the {\it Fermat metric} of $(\R\times S,g)$. Moreover, one can determine if two points are chronologically related using the distance associated to the Fermat metric $d_F$ (see Proposition \ref{teo:caracrel}). Summing up, as causality is determined by the lightlike pre-geodesics and these can be controlled by the geodesics of the Fermat metric \eqref{eq:wsfinsler}, all the causal ladder of $(\R\times S,g)$ can be described in terms of the geodesics and the associated distance of the Fermat metric.  This is done in \S \ref{s:ladder}, particularly, in Propositions \ref{prop:escalera1}, \ref{prop:esc2} and \ref{prop:disting}, Corollary \ref{prop:caucon} and Theorem~\ref{teo:CSYGH}. This characterization is especially fashionable in terms of the symmetrized distance of $d_F$, denoted by $d_s$. In particular, as commented in Remark \ref{dscausality}, an \som-spacetime is {\it totally vicious} if and only if $d_s\equiv-\infty$, {\it chronological} if and only if $d_s\geq 0$ and both possibilities are the only ones for $d_s$. Moreover, it is {\it distinguishing} if and only if $d_s$ is strictly positive away from the diagonal, and {\it globally hyperbolic} if, in addition, the balls of $d_s$ are pre-compact. The causal ladder is completed by observing that if the \som-spacetime is {\it (past or future)-distinguishing}, then it is causally continuous, and it is {\it causally simple} if and only if $(S,F)$ is convex. A different approach to describe the causality of such spacetimes can be found in \cite{Harris15}. In the Appendix, we explain how both approaches are related.

We will exploit these results around four applications. The first one consists of studying the conformal maps of Riemannian Killing submersions of the form $\pi:(\R\times S,g_R)\rightarrow (S,h)$, namely, the conformal maps of the total space that preserve the Killing vector field. It turns out that there exists an \som-spacetime $(\R\times S,g_L)$ with the same conformal maps which preserve the Killing field $K=\partial_t$ as the Killing submersion. Moreover, these conformal maps can be described generalizing \cite[Theorem 4.3]{JLP15} in terms of almost isometries of the Fermat metric, which are the maps that preserve the metric up to the addition of the differential of a function (see Definition \ref{almostisometric} and \S \ref{almoconf}). The second one provides some results of geodesic connectedness  and existence and multiplicity of periodic geodesics of pre-Randers metrics. More precisely, we use the characterization of global hyperbolicity in terms of pre-Randers metrics to deduce that, in such a case, there exists a complete Randers metric with the same pre-geodesics of the given pre-Randers metric. Then we can use the well-known results of connectedness and existence of periodic geodesics for Finsler metrics. In particular, we generalize the celebrated Gromoll-Meyer Theorem and another result of multiplicity of periodic geodesics by Bangert and Hingston (see Theorem \ref{thm:mult}). Beyond compactness, we need and additional condition for pre-Randers metrics to state these results: the length of all the loops with a given basepoint $x\in S$ has to be non-negative. For the result of connectedness, we also need pre-compact symmetrized balls. The third application involves the relation between the structure of \textit{future horizons} in the \som-spacetimes and the so-called \textit{cut locus} $Cut_C$ in $(S,F)$, both associated with a set $C$. Concretely, and by recalling  results by Beem and Krolak \cite{BK98}, and Chrusciel and Galloway \cite{CFGH02} for spacetimes; we obtain a characterization for the uniqueness of C-minimizing segments (see definition in \S \ref{sec:cutlocus}) and the differentiability of the function distance to $C$  associated with $F$. As a consequence, we will be able to prove that the measure of the cut locus $Cut_C$ of $(S,F)$ is zero, where by cut locus $Cut_C$ we mean the set of points where C-minimizing segments do not minimize anymore (extending the Riemannian results in \cite{CFGH02}). The last application is a consequence of the second one. This is because magnetic geodesics with energy $c>0$ are pre-geodesics of a pre-Randers metric (see Proposition \ref{magranders}) and this allows us to obtain results of connectedness and multiplicity for every energy level $c>0$ whose associated pre-Randers metric satisfies the completeness conditions.

The paper is organized as follows. In \S \ref{preliminaries}, we give some basic results about pre-Finsler metrics, showing that its associated pre-distance (which is not necessarily non-negative) is continuous, Proposition \ref{prop:continuidad}. In \S \ref{s:stat}, we introduce the class of \som-spacetimes and its associated Fermat metric, explaining the effect of considering different splittings in \S \ref{sec:almostisometriessplit}, characterizing its causal ladder in \S \ref{s:ladder} and describing its conformal maps in \S \ref{almoconf}. In \S \ref{s:confmapssub}, we characterize the conformal maps of Killing submersions using the almost isometries of a pre-Randers metric (see Corollary \ref{confmapsKillingsub}). In \S \ref{s:presom}, we exploit the applications of the causal ladder to the existence and multiplicity results of periodic geodesics and connectedness of pre-Randers metrics, which are used in \S \ref{magnetic} to get the same type of results for magnetic geodesics. In \S \ref{sec:cutlocus}, we study some results involving the differentiability of the distance function and the measure of the cut locus. Finally, in the Appendix, we describe the relation between the description of the causal ladder given  in \cite{Harris15} by S. Harris, and the results in \S \ref{s:ladder}.

\section{Preliminaries on pre-Finsler Metrics}\label{preliminaries}

Our aim in this first section is to introduce the concept of  pre-Finsler  metrics (see \cite[\S 9.1.2]{SRK14}),  a generalization of the classical notion of Finsler metrics where we remove the property of being positive for all non-zero vector  and the requirements on the fundamental tensor. 

\begin{defi}\label{preFinsler}
Let $S$ be an $(n-1)$-dimensional manifold and $TS$ its tangent bundle.  A function $F:TS\rightarrow \R$ will be called a  {\em pre-Finsler metric}   if:
	
	\begin{itemize}
		\item[(i)] $F:TS\setminus {\bf 0}\rightarrow \R$ is smooth, where $\bf 0$ is the zero section.
		
		\item[(ii)] $F$ is positive homogeneous, that is, $F(\lambda\,v)=\lambda\,F(v)$ for all $v\in TS$ and $\lambda>0$.
		
	\end{itemize}
	 In particular, the last property implies that $F(0)=0$. The pair $(S,F)$ will be called a {\em pre-Finsler} manifold.
	
\end{defi}
%
%
	
Even if, given a point $p\in S$, $F_p$ is not a Minkowski norm in general, we can proceed in complete analogy with the standard Finsler case and re-obtain some of its classical constructions. For instance, given a piecewise smooth curve $c:[t_0,t_1]\rightarrow S$, we can define the {\em length} of $c$ as
\begin{equation}\label{length}
\ell_{F}(c):=\int_{t_0}^{t_1} F(\dot{c}(s))ds. 
\end{equation}
The main difference with the standard case is that the length of the curve could be negative, but we can go on  with the analogy in order to define some kind of distance between points associated with the pre-Finsler metric (compare with \cite{JaSa14}). In fact, 

\begin{defi}\label{def:distance}
	Let $(S,F)$ be a pre-Finsler metric.  We say that a curve $\gamma:[a,b]\rightarrow S$ is a pregeodesic of $F$ if it is a critical point of the length \eqref{length}  in the space $C(x_0,x_1)$,   where $C(x_0,x_1)$  denotes  the set of piecewise smooth curves between $x_0$ and $x_1$, being $x_0=\gamma(a)$ and $x_1=\gamma(b)$. If, in addition, $\gamma$ is  affinely parametrized with the length, we will say that it is a geodesic. Moreover,  we define the associated  pre-distance  as a map $d_{F}:S\times S\rightarrow \R\cup \{-\infty\}$ given by
	
	\[d_F(x_0,x_1)={\rm inf}_{c\in C(x_0,x_1)} {\ell}_F(c).\]
	Observe that $d_F$ is not necessarily non-negative, which leads us to call it pre-distance, but it does satisfy the triangle inequality.
	
\end{defi}
\begin{rem}
 In principle, it is not guaranteed the existence and unicity of geodesics having a fixed vector as inicial velocity, namely, the exponencial map is not necessarily well-defined. Nevertheless, this will not be a problem when the pre-Finsler metric is given by \eqref{eq:wsfinsler}.  
\end{rem}

 There are only two properties that hold for these metrics, the triangle inequality (which follows straightforwardly) and the continuity of $d_{F}$:
%
%
%

\begin{prop}\label{prop:continuidad}
	Let $d_F$ be the associated pre-distance of a pre-Finsler metric $F$ defined on a connected manifold $S$. Then
\begin{enumerate}[(i)]
\item $d_F$ satisfies the triangle inequality (with the usual convention about sums with $-\infty$), 
\item if $\gamma:[a,b]\rightarrow S$ is a piecewise smooth curve satisfying that $d_F(\gamma(a),\gamma(b))=\ell_F(\gamma)$, then it is a pregeodesic,
\item if $d_F(x,x)<0$ for some $x\in S$, then $d_F\equiv -\infty$,
\item if $d_F(x,y)=-\infty$ for some $x,y\in S$, then $d_F\equiv -\infty$,	
\item if $d_F\not=-\infty$, then the function $d_F:S\times S\rightarrow \R$ is continuous.
\end{enumerate}
\end{prop}
\begin{proof}
 Parts $(i)$ and $(ii)$ follow straightforwardly. 
 For $(iii)$, first of all, observe that $d_{F}(x,x)\in\{-\infty,0\}$.  In fact, it is clear from the ``constant'' curve $x$ that $d_{F}(x,x)\leq 0$. Moreover, if $d_{F}(x,x)<0$, then there exists a closed curve $\gamma$ with  basepoint  $x$ and such that $\ell_{F}(\gamma)<0$. In particular, the concatenation of the curve $\gamma$ with itself gives us another closed curve with the same base point and length $2\,\ell_{F}(\gamma)$. By iteration, we deduce that we can construct a curve with length as negative as desired, so $d_{F}(x,x)=-\infty$.  Finally, observe that if $d_{F}(x,x)=-\infty$ for some $x\in S$, then $d_{F}(y,y)=-\infty$ for all $y\in S$.   In fact, as $d_{F}(x,y)$ and $d_{F}(y,x)$ are bounded from above, the triangle inequality ensures that 
  \[
d_{F}(y,y)\leq d_{F}(y,x) + d_{F}(x,x) + d_{F}(x,y)
    \]   and necessarily $d_F(y,y)=-\infty$. For any two arbitrary points $x,y\in S$, we have 
     \[
d_{F}(x,y)\leq d_{F}(x,x) + d_{F}(x,y)
    \]  
    and as $d_F(x,x)=-\infty$, it follows that $d_F(x,y)=-\infty$.
      For  part $(iv)$, use again the triangle inequality $d_F(x,x)\leq d_F(x,x)+d_F(x,y)$ to prove that $d_F(x,x)=-\infty$ and then apply part $(iii)$ to deduce that $d_F\equiv -\infty$.
Let us show $(v)$. 
	 If $d_{F}$ were not continuous, then we could find convergent sequences $x_n\rightarrow x$ and $y_n\rightarrow y$ such that there exists the limit $\lim_{n\rightarrow +\infty} d_{F}(x_n,y_n)\in \R\cup \{-\infty,+\infty\}$ and  $\lim_{n\rightarrow +\infty} d_{F}(x_n,y_n)\not=d_{F}(x,y)$. Observe that we can consider neighborhoods $U_x$ and $U_y$ in the manifold $S$ of $x$ and $y$, respectively, such that they admit a Riemannian metric $g$. Moreover,  there exists a constant $c$ such that $-c|v|\leq F(v) \leq c|v|$ for every $v\in T_qS$ with $q\in U_x\cup U_y$, where $|\cdot|$ is the norm of the Riemannian metric $g$. This easily implies that there exist curves $\beta^1_n$, $\beta^2_n$ from $x_n$ to $x$ and from $x$ to $x_n$, respectively, and $\alpha^1_n$, $\alpha^2_n$ from $y_n$ to $y$ and from $y$ to $y_n$, respectively, such that $\lim_{n\rightarrow +\infty}\ell_{F}(\beta^i_n)=\lim_{n\rightarrow +\infty}\ell_{F}(\alpha^i_n)=0$ for $i=1,2$. Let $\gamma_n:[0,1]\rightarrow S$ be a sequence of curves with $\gamma_n(0)=x_n$ and $\gamma_n(1)=y_n$ and $\ell_{F}(\gamma_n)<d_{F}(x_n,y_n)+1/n$. Then the concatenation $\tilde{\gamma}_n=\alpha^1_n\star\gamma_n\star \beta^2_n$ is a curve from $x$ to $y$ such that $\lim_{n\rightarrow +\infty}\ell_{F}(\tilde{\gamma}_n)=\lim_{n\rightarrow +\infty}d_{F}(x_n,y_n)$. This proves that 
	\begin{equation}\label{firstineq}
	\lim_{n\rightarrow +\infty}d_{F}(x_n,y_n)\geq d_{F}(x,y).
	\end{equation}
	Now let $\rho_n:[0,1]\rightarrow S$ be a sequence of curves from $x$ to $y$ such that $\lim_{n\rightarrow +\infty}\ell_{F}(\rho_n)=d_{F}(x,y)$. The concatenation $\tilde{\rho}_n=\alpha^2_n\star\rho_n\star \beta^1_n$ gives a sequence of curves from $x_n$ to $y_n$ which satisfies that  $\lim_{n\rightarrow +\infty}\ell_{F}(\tilde{\rho}_n)=\lim_{n\rightarrow +\infty}d_{F}(x,y)$. This concludes that $\lim_{n\rightarrow +\infty}d_{F}(x_n,y_n)\leq d_{F}(x,y)$, which together with \eqref{firstineq}, gives a contradiction and it concludes the continuity of $d_{F}$.  
\end{proof}

As in (regular) Finsler theory, the map $d_{F}$ is not necessarily symmetric so it  makes   sense to define the associated {\em symmetrized} pre-distance defined as
\begin{equation}\label{symm}
d_s(x,y)=\frac{1}{2}(d_{F}(x,y)+d_{F}(y,x)).
\end{equation}
From construction, $d_s$ is symmetric, but it is not necessarily positive. 

  Finally we will also extend the concept of {\em almost isometry} (see \cite{JLP15}) for pre-Finsler metrics. 
  \begin{defi}\label{almostisometric}
  We will say that two metrics $F$ and $F'$ are almost isometric if there exists a function $f:S\rightarrow \R$ such that $F=F'+df$. 
  \end{defi}
  In terms of their associated pre-distances, if $F$ and $F'$ are almost isometric, it is not difficult to check that 
\begin{equation}\label{almostDF}
d_{F}(x_0,x_1)=d_{F'}(x_0,x_1)+f(x_1)-f(x_0),
\end{equation} 
for every $x_0,x_1\in S$ (observe that $\ell_{F-df}(c)=\ell_F(c)+f(x_0)-f(x_1)$ for every $c\in C(x_0,x_1)$). 

\section{Stationary-complete Manifolds admitting global splittings and its causality}\label{s:stat}

In this section we will show how the causality of stationary-complete spacetimes admitting a global splitting is characterized in terms of a pre-Finsler metric.

Let us consider $(M,g)$ a stationary-complete spacetime, that is, a Lorentz manifold admitting a globally defined complete timelike Killing vector field $K$  (with the time-orientation provided by the Killing vector field).  We will say that $(M,g)$ admits a  (non-canonical)  global splitting if $M=\R\times S$ for some $(n-1)$-dimensional manifold $S$ and

\begin{equation}\label{def:metrica}
g=-\beta  dt\otimes dt+dt\otimes\omega + \omega\otimes dt + g_{0},
\end{equation}
where $\beta:S\rightarrow (0,+\infty)$ is a positive function, $\omega$ is a one-form of $S$ and $g_{0}$ is a $(0,2)$-symmetric tensor defined over $S$. As we can see, the vector field $K$ is identified with $\partial_t$ and its flow intersects exactly once the hypersurface $S$. 

 Observe that if a stationary-complete spacetime is chronological, then it admits a global splitting (see \cite{Harris92}). 
 \begin{prop}
 A $(0,2)$-tensor in a manifold $M=\R\times S$ as in \eqref{def:metrica} is a Lorentz metric if and only if
\begin{equation}\label{eq:condLor} 
g_{0}(v,v)+\frac{1}{\beta}\omega(v)^2>0
\end{equation} 
for every $v\in TS\setminus \bf 0$, namely,
  $g_0+\frac{1}{\beta}\omega\otimes\omega$ is a  metric on $S$.
 \end{prop}
 \begin{proof}
First observe that $g(\partial_t,\partial_t)=-\beta<0$, and then $g$ is Lorentz if and only if  its $g$-orthogonal  subspace is  $g$-positive definite. Let us take $p=(t,x)\in M$ and consider $T_p M\equiv\R\times T_x S$. Vectors on such a tangent space  are naturally identified with pairs $(\tau,v)\in\R\times T_x S$, having in particular that $\partial_t\equiv (1,0)$. It follows then that $(\tau,v)\in \{\partial_t\}^\perp$ if, and only if, $\tau=(1/\beta) \omega(v)$. Moreover, as
 $g\left((\frac{1}{\beta}\omega(v),v),(\frac{1}{\beta}\omega(v),v)\right) =g_{0}(v,v)+\frac{1}{\beta}\omega(v)^2,$
for every $v\in TS$, the equivalence follows.
\end{proof}
\smallskip

 If, in addition to the previous condition, $g_0$ is Riemannian, we obtain the well-known  class of {\em standard} stationary spacetimes. 
The above class can be characterized as those stationary-complete spacetimes which satisfy the observer manifold condition. Recall that a stationary complete spacetime $(M,g)$  satisfies the observer manifold condition with respect to the Killing field $K$ if for each point $x \in M$ there is a neighborhood $U_x$ of $x$ such that for each point $y\in M$ there is a neighborhood $W_{x,y}$ of $y$ such that for $|t|$ big enough, $\varphi^K_t(U_x) \cap W_{x,y} = \emptyset$. Herem $\varphi^K_t$ denotes the flow of the Killing field $K$. In particular, this condition implies that the Killing orbits are all lines and not circles. As it was observed after \cite[Definition 1.1]{Harris15}, recalling some observations by Palais \cite{Palais61}, the observer manifold condition implies that the space of stationary observers (Killing orbits) is a Hausdorff manifold. Then the metric of the spacetime can be expressed as in \eqref{def:metrica} with the data satisfying \eqref{eq:condLor}. Conversely, it is straightforward to check that all the spacetimes of this form admit a complete timelike Killing vector field and satisfy the observer manifold condition.
 \begin{defi}
 	\label{def:classspacetimes}
 Stationary-complete spacetimes satisfying the observer manifold condition, or equivalently, the spacetimes of the form \eqref{def:metrica} and satisfying \eqref{eq:condLor} will be called  {\som}-spacetimes for short. Moreover, the space of such  spacetimes defined on $\R\times S$ and having as  distinguished timelike Killing vector field $\partial_t$ ($t$ the first coordinate of $\R\times S$) will be denoted as ${\rm \som}(\R\times S)$.
 \end{defi}
Observe that if a stationary-complete spacetime does not satisfy the observer manifold condition, then it is not chronological \cite{Harris92}. As happens with standard stationary spacetimes (see \cite{CJS11}), the chronological relation can be characterized in terms of some ``metric'' structure on $S$. However, given the generality of these models, we will not obtain a Finsler metric as in the standard case, but a pre-Finsler one. Let us see how Fermat principle establishes a bridge between causal and metric concepts.
\begin{prop}\label{prop:Fermatmetric}
Let $(M,g)$ be an \som-spacetime and $\gamma:[a,b]\rightarrow M$, given by $\gamma(s)=(t(s),c(s))$ for all $s\in [a,b]$, a null curve. Then $\gamma$ is a lightlike future-directed geodesic if and only if $c$ is a pregeodesic of the pre-Finsler metric $F$ given by
\begin{equation}\label{eq:wsfinsler} F(v)=\frac{1}{\beta}\omega(v)+\sqrt{\frac{1}{\beta^2}\omega^2(v)+ \frac{1}{\beta}g_0(v,v)}
\end{equation} 
for $v\in TS$, parametrized with $h(\dot c,\dot c)=\frac{1}{\beta^2}\omega^2(\dot c)+ \frac{1}{\beta}g_0(\dot c,\dot c)$ constant. Moreover, 
for any future-directed lightlike curve $\tilde \gamma(s)=(\tau(s),x(s))$,
\begin{equation}\label{Flength}
\tau(s)=\int_a^s F(\dot x(\mu))d\mu.
\end{equation}
\end{prop}
\begin{proof}
 First observe that  $\tilde\gamma$ is null if, and only if,
\[
0= g(\dot{\tilde\gamma}(s),\dot{\tilde\gamma}(s))=-\beta \dot{\tau}(s)^2+2\omega(\dot{x}(s))\dot{\tau}(s) + g_0(\dot{x}(s),\dot{x}(s)).
\]
Hence, if in addition, $\tilde \gamma$ is future-directed, it follows that
$\dot{t}(s)=F(\dot{x}(s))$ and then \eqref{Flength}. 
Using the Fermat Principle as for example in \cite[Theorem 4.1]{CJM11}, we conclude.
\end{proof}

Last proposition shows in particular that it is possible to define an exponential map for the pre-Finsler metric given in \eqref{eq:wsfinsler}. Indeed, in every direction $v\in TS\setminus 0$, one can consider the lightlike geodesic with initial velocity $(1/\beta \omega(v),v)$, which projects to a pre-geodesic of $F$. That is, the exponential of $F$ is roughly speaking the projection of the exponential of $g$ restricted to the lightlike cone. Observe however that defined in this way, this exponential map does not have any information about the affine parametrization of geodesics.  We will call the metric given in \eqref{eq:wsfinsler} the Fermat metric associated with $(M,g)$. Let us introduce the class of {\em pre-Randers metrics} in $S$ as those metrics $F:TS\rightarrow \R$ given by 
\begin{equation}\label{preRanders}
F(v)=\sqrt{h(v,v)}+\omega(v),
\end{equation}
for every $v\in TS$, where  $h$ and $\omega$ are, respectively, a Riemannian metric and a one-form in $S$. Observe that the only difference with the classical Randers metrics is that  no condition is required on  the one-form $\omega$ and, as a consequence, $F$ can take  negative values for some vectors. If we denote by 
 ${\rm pRand}(S)$ the subset of pre-Randers metrics on $S$, 
  then we have a map
\begin{equation*}
\varphi:{\rm \som }(\R\times S)\rightarrow {\rm pRand}(S),
\end{equation*}
where $\varphi(g)=F$ is the Fermat metric associated with $g$. This map is not injective (conformal metrics have the same image), but it is surjective, as given a pre-Randers metric as in \eqref{preRanders}, we can define an \som-spacetime, with $g_0=h-\omega\oplus\omega$ and $\beta=1$ (which clearly satisfies \eqref{eq:condLor}). As we will show soon, the stationary-complete spacetimes will be a great tool to study pre-Randers metrics, as well as pre-Randers metrics allow one to study the causality of the spacetime. 

 We will say that two
events $p$ and $q$ in a spacetime are
chronologically related, denoted $p\ll q$ (resp. strictly
causally related $p< q$)  if there exists a future-directed
timelike (resp.
 causal) curve $\gamma$ from $p$ to $q$; $p$ is causally related to $q$ if either $p<q$ or $p=q$, denoted $p\leq q$.
Then the {\it chronological future}  (resp. {\it causal future}) of
$p\in M$ is defined as $I^+(p)=\{q\in M : p\ll q\}$ (resp.
$J^+(p)=\{q\in M : p\leq q\}$).  We define the chronological and causal ``past"  analogously, denoting them
$I^-(p), J^-(p)$, respectively.  We will need to introduce the balls associated with a pre-Finsler metric $F$. The forward (resp. backward) ball associated with $F$ of center $x_0\in S$ and radius $r\in \R$ is defined as $B^+_F(x_0,r)=\{x\in S:d_F(x_0,x)<r\}$ (resp. $B^-_F(x_0,r)=\{x\in S:d_F(x,x_0)<r\}$, where $d_F$ is the pre-distance associated with $F$ (see Definition \ref{def:distance}).  The chronological relation is then characterized by:

\begin{prop}\label{teo:caracrel}
	Let $(M,g)$ be a spacetime as in \eqref{def:metrica} satisfying \eqref{eq:condLor}. Then, 
	\begin{enumerate}[(i)]
	\item  a vector $(\tau,v)\in TM\equiv \R\times TS$ is lightlike and future-directed if and only $\tau=F(v)$, 
	\item  a vector $(\tau,v)\in TM\equiv \R\times TS$ is timelike and future-directed if and only $\tau>F(v)$, 
	\item if $d_{F}$ is the pre-distance associated with $F$, then
	\begin{enumerate}[(a)]
		\item $(t_0,x_0)\ll (t_1,x_1)\iff d_{F}(x_0,x_1)<t_1-t_0$,
 \item $(t_0,x_0)\leq  (t_1,x_1)	 \Longrightarrow d_{F}(x_0,x_1)\leq t_1-t_0$,
\item  $I^+(t_0,x_0)=\bigcup_{s\in\R}\{t_0+s\}\times B_F^+(x_0,s)$,
	  \item   $I^-(t_0,x_0)=\bigcup_{s\in\R}\{t_0-s\}\times B_F^-(x_0,s)$. 
	    \end{enumerate} 
 	\end{enumerate}
\end{prop}

\begin{proof}
	Part $(i)$ and $(ii)$ are straightforward and the proof of part $(iii)$  follows the same ideas as the one in \cite[Proposition 4.2]{CJS11}. 
	%
	%
	%
\end{proof}

Observe that  previous characterization is essentially the same  as the one obtained  in the standard case   (see \cite[Proposition 4.2]{CJS11}).  The main difference is that $d_{F}$ can take negative values, and so, it could happen that $(t_0,x_0)\ll (t_1,x_1)$ even if $t_1<t_0$.

That difference has implications even at the topological level. In fact, in the standard case, the balls of the distance $d_{F}$ determine the topology on $S$. However, in the non-standard case, that is no longer true. 

\subsection{On different splittings of \som-spacetimes}

\label{sec:almostisometriessplit}
Let us  investigate the relation between the pre-Finsler metrics corresponding to different splittings of the same \som-spacetime.   To do so,  consider $(M,g)$ a spacetime as in \eqref{def:metrica} satisfying \eqref{eq:condLor} and $S'\subset M$, a hypersurface  that  intersects exactly once every integral curve of the timelike Killing field (namely, the vertical lines). From now on, a hypersurface $S'$ with this property will be called a slice of $\R\times S$. Then, we can consider the smooth function $f:S\rightarrow \R$ characterized by the fact that $(f(x),x)\in S'$, so let us denote $S'=S_f$.  Conversely, every smooth function $f:S\rightarrow\R$ determines a slice $S_f=\{(f(x),x)\in\R\times S: x\in S\}$. For each one of these slices, one can obtain a different splitting $\R\times S_f$ of $M$ using the flow of the Killing vector field. If we define the map $\varphi: \R\times S\rightarrow \R\times S$, $\varphi(t,x)=(t+f(x),x)$, the metric $\varphi^*g$ is isometric to $g$, as $\varphi:(\R\times S,\varphi^*(g))\rightarrow (\R\times S,g)$ is an isometry, and it is also an \som-spacetime. 

\begin{prop}\label{diffsplittings}
Let $(\R\times S,g)$ be an \som-spacetime. Given an arbitrary function $f:S\rightarrow \R$,  the pullback metric $g^f=\varphi^*g$ is expressed as 
\[
g^f=-\beta dt\otimes dt+\omega^f\otimes ds+ds\otimes\omega^f + g^f_0 
\]	
where $\omega^f$=$\omega-\beta df$ and $g^f_{0}=g_{0}+df\otimes \omega+\omega\otimes df -\beta df\otimes df$. Moreover, $(\R\times S,\varphi^*g)$ is also an \som-spacetime and its  Fermat metric is given by $F^f=F-df$, where $F$ is the Fermat metric of $(\R\times S,g)$.
\end{prop}
\begin{proof}
	It follows the same lines as the proof of \cite[Prop. 5.9]{CJS11}.
\end{proof}

Recall if the difference of two Finsler metrics equals the differential of a function, then we say that they are almost isometric (see Definition \ref{almostisometric}). 
 \begin{cor}\label{cordiffsplittings}
	Two different pre-Randers metrics are associated with different splittings of the same \som-spacetime with the same Killing field $\partial_t$ if and only if they are almost isometric.
	\end{cor}
\smallskip 
 The last corollary implies that almost isometric pre-Randers metrics must have many properties in common as they share the same Lorentzian manifold. 
As a final remark on this section, let us just recall that all the splittings over the same stationary-complete spacetime also share the same symmetrized pre-distance.

\begin{lemma}\label{lem:ind}
	 The symmetrized pre-distance $d_s$ given in \eqref{symm}  associated with the Fermat metric $F$ of a Lorentzian splitting $(\R\times S,g)$ as in \eqref{def:metrica}  is invariant by a change of splitting of $(\R\times S,g)$ as in Proposition \ref{diffsplittings}. 
\end{lemma}
\begin{proof}
	 Let us consider $(\R\times S,g)$ and $(\R\times S,g^f)$ two given splittings of the stationary-complete spacetime $(M,g)$ with associated pre-Finsler metrics $F$ and $F^f$. From Corollary \ref{cordiffsplittings}, both pre-Finsler metrics are almost isometric, and so, there exists a function $f:S\rightarrow \R$ ensuring that $F=F^f+df$. 
	 From \eqref{almostDF}, it follows that:
	\[d_{F}(x_0,x_1)= d_{F^f}(x_0,x_1)+f(x_1)-f(x_0),\]
	hence the symmetrization given by \eqref{symm} of $d_{F}$ coincides with the symmetrization of $d_{F^f}$, as desired.
	%
\end{proof}
%


\subsection{The causal ladder}\label{s:ladder}

The aim in this subsection is to characterize the causal global structure of a spacetime $(M,g)$ as in \eqref{def:metrica} (satisfying \eqref{eq:condLor}) in terms of the associated $d_{F}$, by using essentially the characterization given in Prop. \ref{teo:caracrel}. Let us start with the first two steps of the causal ladder

\begin{prop}\label{prop:escalera1}
	Let $(M,g)$ be a spacetime as in \eqref{def:metrica}. Then $(M,g)$ is:
	\begin{itemize}
		\item[(C1)]   totally vicious if and only if there exists a closed curve $\gamma$ such that $\ell_F(\gamma)<0$ if and only if $\exists  \,x\in S$ such that $d_F(x,x)<0$ (and then $d_F\equiv -\infty$).                   
		\item[(C2)] chronological if and only if $d_{F}(x,x)=0$ for  some (and then, all) $x\in S$ if and only if $d_s(x,y)\geq 0$ for every $x,y\in S$. 
	\end{itemize}
\end{prop} 
\begin{proof}

  From Prop. \ref{teo:caracrel}, $(t,x)\ll (t,x)$ if and only if $d_{F}(x,x)<0$. Then, $p\not\ll p$ for some   $p=(t,x)\in M$ if, and only if,  $d_{F}(x,x)=0$  and both, (C1) and (C2) follow  using part $(iii)$ of Proposition \ref{prop:continuidad} and the triangle inequality $d_F(x,x)\leq d_F(x,y)+d_F(y,x)=2 d_s(x,y)$.
\end{proof}


\begin{prop}\label{prop:esc2}
	$(M,g)$ is causal if and only if  $d_{F}(x,x)=0$ for all $x\in S$ and  there is no non-trivial closed curve $c$  with $\ell_F(c)=0$ in $S$.
\end{prop}

\begin{proof}
The equivalence follows from part (C2) in Proposition \ref{prop:escalera1} and the fact that any  closed  curve $c$  defines a null closed curve $\gamma(s)=(t(s),c(s)):[0,1]\rightarrow \R\times S$ with $\dot t= F(\dot c)$ if and only if with $\ell_{F}(c)=0$ (and so $0=t(1)-t(0)=\int_0^1F(\dot{c})ds$, recall Proposition \ref{prop:Fermatmetric}). 
%
%
%
\end{proof}

%
%
%
%
%

\smallskip 

For the following step on the ladder (that is, past and future distinguishing), we will need first to characterize the future and the past of points in $(M,g)$. Let $p_0=(t_0,x_0)\in M$ be an arbitrary point and observe that
%

\begin{equation}\label{eq:characpasado} 
\begin{array}{rl}
I^-(p_0) & = \{(t,x)\in M: (t,x)\ll (t_0,x_0)\}\\
& = \{(t,x)\in M: d_{F}(x,x_0)<t_0-t\}\\
& = \{(t,x)\in M: t<t_0-d_{F}(x,x_0)\}.
\end{array}
\end{equation}

In conclusion, and in a complete analogy with the standard case, the past of a point is determined by the function $d_{p_0}^+(\cdot):=t_0-d_{F}(\cdot,x_0)$. 

\begin{lemma}\label{lem:chroncontained}
	 $I^+(p_0)\subseteq I^+(p_1)$  (resp. $I^-(p_0)\subseteq I^-(p_1)$) if and only if $d_{F}(x_1,x_0)\leq t_0-t_1$ (resp. $d_{F}(x_1,x_0)\leq  t_1-t_0 $).
\end{lemma}
\begin{proof}
	 Observe that as $\partial_t$ is timelike and the chronological condition is open, $I^+(p_0)\subseteq I^+(p_1)$ if and only if $(t_0 + \epsilon,x_0)\in I^+(p_1)$ for every $\epsilon>0$, but the last condition is equivalent to $d_{F}(x_1,x_0)\leq t_0-t_1$ (the other equivalence is proved analogously). 
\end{proof}

This information is enough to effectively characterize when $(M,g)$ is past-distinguishing (being future-distinguishing analogous):

\begin{prop}\label{prop:disting}
 	$(M,g)$ is past-distinguishing (resp. future-distinguishing) if and only if $d_s$ is a distance, namely, for all $x_0,x_1\in S$ with $x_0\neq x_1$, $d_s(x_0,x_1)>0$. 
\end{prop}
\begin{proof}
  From Lemma \ref{lem:chroncontained}, we deduce that  $I^-(p_0)=I^-(p_1)$ if and only if $d_{F}(x_0,x_1)\leq t_1-t_0$ and $d_{F}(x_1,x_0)\leq t_0-t_1$. Then if $(M,g)$ is not past-distinguishing, there exist $p_0=(t_0,x_0)$  and $p_1=(t_1,x_1)$, $p_0\not=p_1$, such that $d_{F}(x_0,x_1)\leq t_1-t_0$ and $d_{F}(x_1,x_0)\leq t_0-t_1$, which implies that $d_s(x_0,x_1)\leq 0$.  If $x_0\neq x_1$, we deduce that $d_s$ is not a distance, so assume that  $x_0=x_1$.  As $p_0\neq p_1$, $t_0\not=t_1$ and this implies that $d_s(x_0,x_0)=d_{F}(x_0,x_0)\leq \min \{t_0-t_1,t_1-t_0\}<0$. By part $(iii)$ of Proposition \ref{prop:continuidad}, $d_F=d_s=-\infty$ and therefore $d_s$ is not a distance.

  Assume now that $d_s$ is not a distance. Then there exist $x_0,x_1\in S$, $x_0\not=x_1$, such that $d_s(x_0,x_1)\leq 0$. This implies that  $\min\{d_{F}(x_0,x_1),d_{F}(x_0,x_1)\}\leq 0$. Assume without loss of generality that  $d_{F}(x_1,x_0)\leq 0$ and let $t_1=-d_{F}(x_1,x_0)$. Then if $p_0=(0,x_0)$ and $p_1=(t_1,x_1)$, we  have  that $d_{F}(x_1,x_0)=-t_1$ and as $d_s(x_0,x_1)\leq 0$, $d_{F}(x_0,x_1)\leq -d_{F}(x_1,x_0)=t_1$. By Lemma \ref{lem:chroncontained}, $I^-(p_0)=I^-(p_1)$, namely, $(M,g)$ is not
	past-distinguishing.

	%
	%
	%
	%
\end{proof}

\medskip
%
The following three steps on the causal ladder, that is, strongly causal, stably causal and causally continuous are consequences of the past-distinguishing condition. In fact:

\begin{cor}\label{prop:caucon}
	$(M,g)$ is causally continuous if and only if it is   (future or past)  distinguishing.
\end{cor}
\begin{proof}
	 Recall that, by Proposition \ref{prop:disting}, $(M,g)$ is distinguishing if and only if it is future or past distinguishing. Then the equivalence follows from  \cite[Theorem 1.2]{JS08} . 
\end{proof}

The rest of the steps of the causal ladder follow essentially in the same way as in the standard case, with a small variation in the characterization of global hyperbolicity. In fact,

\begin{thm}\label{teo:CSYGH}
	Let $(M,g)$ be a spacetime as in \eqref{def:metrica} satisfying \eqref{eq:condLor}. Then:
	
	\begin{itemize}
		\item[(CS)]  the following are equivalent:
		\begin{enumerate}[(i)]
			\item $(M,g)$ is causally simple,
			\item $J^+(p)$ is closed for all $p\in M$,
			\item $J^-(p)$ is closed for all $p\in M$,
			\item  $(S,F)$ is convex\footnote{That is, for any pair of points $x_1,x_2\in S$ there exists a curve $c:[0,1]\rightarrow S$ with $c(0)=x_1, c(1)=x_2$ and $\ell_{F}(c)=d_{F}(x_1,x_2)$.}.
		\end{enumerate}
		
		\item[(GH)] the following are equivalent:
		
		\begin{enumerate}[(i)]
			\item $(M,g)$ is globally hyperbolic 
			\item $d_s$ is a distance and the balls $B_s(x,r)$ are precompact for all $x\in S$ and $r>0$,
			\item {\cambios$d_{s}$ is a distance and it is satisfied:} (a) for any pair of points $x_0,x_1\in S$ and values $r,s \in\R$, the intersection $\overline{B}^-(x_0,r)\cap \overline{B}^+(x_1,s)$ is compact and (b) if for some $x_0 \in S$  there exists a sequence $\{x_n\}_{n\in\N}\subset S$ with $\lim_{n} d_{F}(x_0,x_n)=-\infty$ (resp. $\lim_{n} d_{F}(x_n,x_0)=-\infty$) then $\lim_{n}d_{s}(x_0,x_n)=\infty$.
                          
		\end{enumerate}
	\end{itemize}
\end{thm}

\begin{proof}
		As the distance $d_F$ is continuous (see Proposition \ref{prop:continuidad}), the equivalences in part $(CS)$ can be proved   in an analogous way  to \cite[part (a) of Proposition 4.3]{CJS11}. 
		
	
		 Consider then the equivalences in $(GH)$ and  let us prove first the equivalence between $\it (i)$ and $\it (ii)$. If $(M,g)$ is globally hyperbolic, then it is distinguishing and, by Proposition \ref{prop:disting}, $d_s$ is a distance. Moreover, by \cite[Theorem 1.2]{JS08}, $(M,g)$ admits a standard splitting and as $d_s$ does not depend on the splitting by Lemma \ref{lem:ind},   \cite[Theorem 4.3 (b)]{CJS11} concludes that the balls of the symmetrized distance $d_s$ are precompact. Conversely, if we assume that $d_s$ is a distance, by Proposition \ref{prop:disting}, we deduce that it is distinguishing and by \cite[Theorem 1.2]{JS08}, it admits a standard splitting. The invariance of $d_s$ by a change of the splitting and part $(b)$ of Theorem 4.3 in \cite{CJS11} conclude that $(M,g)$ is globally hyperbolic.
		
                 For the equivalence between $\it (i)$ and $\it (iii)$, assume first that $\it (i)$ holds. {\cambios As $(M,g)$ is globally hyperbolic, it is in particular past-distinguishing. Then, from Prop. \ref{prop:disting}, it follows that $d_{s}$ is a distance. The proof of $\it (a)$ follows} analogously to the proof of part $\it (b)$ in \cite[Theorem 4.3]{CJS11}. 
		%
		For $\it (b)$, let us assume by contradiction that there exists $x_0$ and a sequence $\{x_n\}\subset S$ such that $\lim_n d_{F}(x_0,x_n)=-\infty$ but $d_{s}(x_0,x_n)$  is bounded for all $n\in\N$ (the case with $\lim_n d_{F}(x_n,x_0)=-\infty$ is completely analogous). Observe that, by Proposition \ref{teo:caracrel},  $(0,x_0)\leq (d_{F}(x_0,x_n),x_n)$ for all $n\in\N$, i.e., the non-convergent sequence $\sigma:=\{(d_{F}(x_0,x_n),x_n)\}_{n\in\N}\subset J^+((0,x_0))$. If  $2d_{s}(x_0,x_n)=d_{F}(x_0,x_n)+d_{F}(x_n,x_0)<t_1$  for some constant $t_1$ and all $n\in\N$, it follows, by using again Proposition \ref{teo:caracrel},  that $(d_{F}(x_0,x_n),x_n)\ll (t_1,x_0)$, i.e., $\sigma\subset J^-((t_1,x_0))$. In conclusion, the sequence $\sigma$ is contained in $J^+((0,x_0))\cap J^-((t_1,x_1))$ but it has no convergent subsequence, a contradiction with the compactness of the latter. Therefore, $\it (b)$ follows.   
		
		\smallskip 
		
		{\cambios Now assume that $d_{s}$ is a distance and} that $\it (a)$ and $\it (b)$ are satisfied, and let us show that $(M,g)$ is globally hyperbolic. Consider $(t_0,x_0),(t_1,x_1)\in M$ two points and consider $J^+((t_0,x_0))\cap J^-((t_1,x_1))$. {\cambios As $d_{s}$ is a distance, then Prop. \ref{prop:disting} plus Cor. \ref{prop:caucon} ensures that $(M,g)$ is causally continuous. Hence, we only need to show that the  previous causal diamond is pre-compact in order to prove that $(M,g)$ is globally hyperbolic.}
                For this, let us consider a sequence $\{(t_n,x_n)\}_{n\in\N}\subset J^+((t_0,x_0))\cap J^-((t_1,x_1))$ and let us show that the sequence admits a convergent subsequence. The first step is to prove that $\{t_n\}_{n\in\N}$ admits a convergent subsequence. Otherwise, we can assume that $\{t_n\}\rightarrow -\infty$ (the case $+\infty$ is completely analogous). As $(t_n,x_n)\in J^+((t_0,x_0))\cap J^-((t_1,x_1))$ for all $n\in\N$, and, according to Proposition \ref{teo:caracrel}, we obtain the following inequalities:
		\[
		\begin{array}{l}\label{eq:eq2}
		d_{F}(x_0,x_n)\leq t_n-t_0\\d_{F}(x_n,x_1)\leq t_1-t_n,
		\end{array}
		\]
		for all $n\in\N$.
		As $\lim_{n\rightarrow +\infty }t_n=-\infty$ , and considering the first inequality, we can observe that $\lim_{n\rightarrow+\infty} d_{F}(x_0,x_n)=-\infty$. Now, from  both inequalities,  we have that for all $n$:
		\[
		d_{F}(x_n,x_1)\leq t_1-t_n\leq t_1-d_{F}(x_0,x_n)-t_0, 
		\]
		and so,
		\[
		d_{F}(x_n,x_1)+ d_{F}(x_0,x_n)\leq t_1-t_0,
		\]
		which contradicts $\it (b)$ as $\lim_{n\rightarrow +\infty} d_{s}(x_n,x_0)=+\infty$ and
		\[\begin{array}{rl}
		2d_s(x_n,x_0)=& d_{F}(x_n,x_0)+ d_{F}(x_0,x_n)\\ \leq & d_{F}(x_n,x_1)+d_{F}(x_1,x_0)+d_{F}(x_0,x_n)\\ \leq & t_1-t_0+d_{F}(x_1,x_0).
		\end{array}\]
		Therefore, we deduce that there exists a convergent subsequence of $\{t_n\}_{n\in\N}$ and we can prove that $(M,g)$ is globally hyperbolic analogously to the proof of \cite[Theorem 4.3 (b)]{CJS11}. 

\end{proof}

\begin{rem}\label{dscausality} Several observations are in order:

 (1) the symmetrized distance $d_s$ allows us to characterize the following causal conditions:
\begin{enumerate}[(i)]
\item $(M,g)$ is totally vicious if and only if $d_s\equiv-\infty$  (part $(C1)$ of Proposition \ref{prop:escalera1}), 
\item $(M,g)$ is chronological if and only if $d_s\geq 0$ (part $(C2)$ of Proposition \ref{prop:escalera1}), ,
\item $(M,g)$ is distinguishing (and then causally continuous) if and only if $d_s(x,y)\geq 0$ for every $x,y\in S$ with strict inequality when $x\not=y$ (Proposition \ref{prop:disting} and Corollary \ref{prop:caucon}), 
\item $(M,g)$ is globally hyperbolic if and only if  $d_s$ is a distance and the balls of $d_s$ are precompact (equivalences $\it (i)$ and $\it (ii)$  of Theorem \ref{teo:CSYGH} $(GH)$. 
\end{enumerate}
Nevertheless, causality and causal simplicity cannot be characterized in terms of $d_s$. Let us recover the examples given in \cite{JS08}. The first example consists in a chronological but not causal spacetime. Consider in $\R^2$ the Lorentzian metric $g_1$  for which  $X=\partial_t-\partial_x$ and
$Y=\partial_x$ are lightlike vectors that satisfy $g_1(X,Y)=-1$. The 
cylinder $C=\R\times S^1$ obtained by identifying $(t,x)\sim
(t,x+1)$ is non-causal, but chronological, and it admits the projection $K$ of $\partial_t$ as a stationary
vector field. Moreover, if we consider the product $C\times \R$, $g_2=g_1+dy^2$, and choose any
irrational number $a$  identifying $(t,x,y)\sim
(t,x,y+1)$ and  $(t,x,y)\sim
(t,x+1,y+a)$, the quotient spacetime $(\tilde M, \tilde g_2)$ is causal. In both cases, it is not difficult to see that $d_s\equiv 0$, in fact, $d_F\equiv 0$.   As causal simplicity is equivalent to the convexity of $F$, it does not admit a characterization in terms of $d_s$. 

(2) 	In spite of what happens  in  the standard case, the characterization of the global hyperbolicity  given in terms of the intersections of the balls  requires an additional condition $\it (b)$. That condition is automatically satisfied whenever $d_{F}(x,\cdot)$ or $d_{F}(\cdot,x)$ is bounded from below, as happens in the  classical  Finsler case where the distance is always positive.
\end{rem}

\subsection{Almost isometries and conformal maps}\label{almoconf} As it was commented in \S \ref{preliminaries}, an almost isometry $\varphi:(M,F)\rightarrow (M,F)$ of a pre-Finsler metric $F$ on $M$ (see Definition \ref{almostisometric}) is also an almost isometry for its associated pre-distance $d_F$ (see \eqref{almostDF}). Unfortunately, the converse is not true. A counterexample is provided by the Fermat metric associated with the non-distinguishing \som-spacetime $(\tilde M, \tilde g_2)$ given in Remark \ref{dscausality}. Observe that $\tilde M=\R\times S$  with $S= T^2$  and $\tilde g_2$ is the projection in the quotient spacetime of
\[g_2=-2dt\otimes dt-dx\otimes dt-dt\otimes dx+dy\otimes dy.\]
It follows that $F=\frac{1}{2}dx+\sqrt{\frac 14 dx^2+\frac 12 dy^2}$. The pre-distance associated with $F$ is $d_F=0$, since the $F$-length of the integral curve of $-\partial_x$ is zero  because  $F(-\partial_x)=0$, but this integral curve is dense on $S$ (recall that the pre-distance is continuous, Prop. \ref{prop:continuidad}).  Therefore, every bijection $\varphi:(S,d_F)\rightarrow (S,d_F)$ is an almost isometry for $d_F$, but not necessarily an almost isometry for $F$. As a consequence, $d_F$ is not a length space,  and then  the proof of  \cite[Lemma 3.1]{JLP15} does not work for pre-Finsler metrics in order to show that an almost isometry of $F$ is an isometry for the symmetrized pre-Finsler metric $\hat F$. There are also problems in  \cite[Proposition 3.3]{JLP15}, where it is used that the triangular function is continuous as  the triangular function cannot be used to characterize almost isometries  of pre-Finsler metrics.  However, \cite[Theorem 4.3]{JLP15} does generalize to pre-Finsler metrics:
\begin{thm}\label{confalmost}
Let $(\R\times S,g)$ be an \som-spacetime.
A map  $\psi: (\R\times S,g)\rightarrow (\R\times S,g)$ is conformal if and only if $\psi(t,x)=(t+f(x),\varphi(x))$ for all $(t,x)\in\R\times S$, with $f:S\rightarrow \R$ an arbitrary smooth function and $\varphi:(S,F)\rightarrow (S,F)$, an almost isometry of $(S,F)$, being $F$ the pre-Finsler metric associated with $(\R\times S,g)$.
\end{thm}
\begin{proof}
Analogous to \cite[Theorem 4.3]{JLP15}.
\end{proof}
 
\subsection{Causal Boundary of stationary spacetimes and pre-Randers metrics}  In \cite{FHSBuseman}, the c-completion of standard stationary
spacetimes was obtained in terms of the Busemann completion on the
associated Randers manifold. As the stationary spacetime must be strongly causal in order to have a consistent definition of causal boundary, it does not seem very interesting 
to consider pre-Randers metrics to study the c-completion, since in such a case, there is always a proper Randers metric available. In any case, and as happens in the standard case, the causal boundary suggests  several  constructions for pre-Randers metrics which are, in itself, of interest.  More precisely,  it can be possible to define the Cauchy completion and the Gromov compactification for pre-Randers metrics. Moreover, in analogy with the standard case, it seems feasible to define an analog for the Busemann completion endowed with a topology inherited from the causal boundary. 





\section{Conformal maps of Killing submersions} \label{s:confmapssub}
 Recall that in Riemannian geometry a Killing submersion is a Riemannian submersion $\pi:(E,g_R)\rightarrow (S,h)$ such that the fibers are one-dimensional and coincide with the orbits of a Killing vector field (see \cite{Man14}). Let us consider the case in that $E=\R\times S$, and the metric $g_R$ in $(x,t)\in \R\times S$ is given  by
\begin{equation}\label{e1Rie} g_R((\tau,v),(\tau,v))=\bar{g}_0(v,v)+2 \bar\omega(v) \tau+\bar\beta \tau^2,
\end{equation}
for $(\tau,v)\in \R\times T_xS$, where $\bar\beta$ is a positive smooth real function
 on $\R\times S$ and $\bar{g}_0$ and $\bar\omega$ are respectively  a Riemannian metric and a one-form on $S$. Moreover, \eqref{e1Rie} is a Riemannian metric if and only if 
\begin{equation}\label{conditionRie}
\bar{g}_0(v,v)-\frac{\bar\omega(v)^2}{\bar\beta}>0
\end{equation} 
for every $v\in TS\setminus 0$.  Observe that we can define an associated \som-spacetime $(\R\times S,g_L)$ with 
\begin{equation}\label{eq:5}
g_L((\tau,v),(\tau,v))=g_R((\tau,v),(\tau,v))-2 \frac{g_R((v,\tau),K)^2}{g_R(K,K)},
\end{equation}
where $K=(1,0)$ is the Killing field. Indeed,  it is straightforward to check that
\[g_L((\tau,v),(\tau,v))=\bar g_0(v,v)-\frac{2}{\bar\beta} \bar{\omega}(v)^2-2\tau \bar{\omega}(v)-\bar\beta \tau^2,\]
and therefore, this is a metric of the type \eqref{def:metrica} with $\beta=\bar\beta$, $\omega=-\bar{\omega}$ and $g_0=\bar g_0-\frac{2}{\beta}\bar\omega\otimes\bar\omega$ and \eqref{eq:condLor} is satisfied thanks to \eqref{conditionRie}. Roughly speaking we obtain a metric with the same values in the orthogonal subspace to the Killing field $K$ and such that the Killing field is still orthogonal to this subspace but with $g_L(K,K)=-g_R(K,K)$.

 Conversely, if we have a \som-metric $g_{L}$ as in \eqref{def:metrica} satisfying \eqref{eq:condLor}, we can define a Killing submersion $(S\times\R,g_R)$ with
\begin{equation}\label{eq:6}
g_R((\tau,v),(\tau,v))=g_L((\tau,v),(\tau,v))-2 \frac{g_L((\tau,v),K)^2}{g_L(K,K)},
\end{equation}
where $K=(1,0)$. Again
\[g_R((\tau,v),(\tau,v))=g_0(v,v)+\frac{2}{\beta}\omega(v)^2-2\tau\omega(v)+\beta\tau^2,\]
therefore $\bar\beta=\beta$, $\bar\omega=-\omega$ and $\bar{g}_0=g_0+\frac{2}{\beta}\omega\otimes\omega$. Moreover, \eqref{conditionRie} is satisfied because of \eqref{eq:condLor}.

Let us denote by ${\rm Stat}_R(\R\times S)$ the space of Riemannian metrics on $\R\times S$ having $K=\partial_t$ as a Killing field, or equivalently, which can be written as \eqref{e1Rie}. The above discussion can be summarized as follows.  

\begin{prop}\label{conformalmaps}
  Let us define  the  map
  \begin{equation}
    \label{eq:4}
    \begin{array}{cccc}
      \Psi: &{\rm\som}(\R\times S) & \rightarrow & Stat_{R}(\R\times S)\\
       & g_L & \rightarrow & g_{R}
    \end{array}    
  \end{equation}
  given by 
  \eqref{eq:6}. Then this maps is one-to-one, with inverse given by \eqref{eq:5} and
  \begin{itemize}
  \item[(i)] the metrics $g_{L}$ and $g_{R}$  share the same Killing vector field $K=\partial_t$ but with $g_L(K,K)=-g_R(K,K)$,
  \item[(ii)] both metrics have the same values in the orthogonal subspace to $K$,
  \item[(iii)] the sets of conformal maps of $g_R=\Psi^*(g_L)$ and $g_L$  which preserve the Killing vector field $K$ coincide. 
  \end{itemize}
\end{prop}
\begin{proof}
  Assertions (i) and (ii) are direct consequences of how the map $\Psi$ is defined by \eqref{eq:5} and \eqref{eq:6}. For (iii), just observe that this property follows from the fact that both metrics $g_R=\Psi^*(g_L)$ and $g_L$ coincide in the orthogonal space to $K$ by part $(ii)$ and they coincide up to the sign on the subspace generated by $K$ which, by hypothesis, is preserved by the conformal map. 
\end{proof}
\begin{cor}\label{confmapsKillingsub}
Given a Riemannian manifold $(\R\times S,g_R)$ with $g_R\in {\rm Stat_R}(\R\times S)$, the conformal maps of $(\R\times S,g)$ which preserve the Killing field $K$ are given by
$\psi:\R\times S\rightarrow \R\times S$, with $\psi(t,x)=(t+f(x),\varphi(x))$, and $\varphi:S\rightarrow S$ is an almost isometry of $(S,F)$, being 
\[F(v)=-\frac{1}{\beta}\bar \omega(v)+\sqrt{-\frac{1}{\beta^2}\bar\omega(v)^2+\frac{1}{\beta}\bar g_0(v,v)}.\]
\end{cor}
\begin{proof}
It follows from part $(iii)$ of Proposition \ref{conformalmaps} and Theorem \ref{confalmost}.
\end{proof}
Compared this result with \cite[\S 2.3]{Man14}, where Killing isometries are studied under a different point of view.

\section{Applications of the pre-Randers/\som-spacetimes correspondence}\label{s:presom}

As we have already mentioned on \S \ref{sec:almostisometriessplit}, the correspondence between pre-Randers metrics  and \som-spacetimes allows us to obtain interesting results. We will make use particularly of the well-known properties of \som-spacetimes in order to infer properties of the associated pre-Randers space.

For instance, by using the characterization on the existence of a standard splitting (see \cite{JS08}), it follows: 
%

\begin{prop}\label{prop5.1}
	Let $(S,F)$ be a pre-Randers metric. If for every two distinct points $x_0,x_1\in S$ holds $d_{F}(x_0,x_1) >-d_{F}(x_1,x_0)$, then $F$ is almost isometric to a (positive) Randers metric. 
\end{prop} 

\begin{proof}
	The hypothesis ensures that the corresponding \som-spacetime is distinguishing (Prop. \ref{prop:disting}), and so, causally continuous (Prop. \ref{prop:caucon}). According to \cite[Theorem 1.2]{JS08}, $(M,g)$ admits then a splittting $(\R\times S',g)$ as a standard stationary spacetime with an associated (positive) Randers metric $F'$. Finally, as $\R\times S'$ and $\R\times S$ are splittings of the same spacetime  with the same Killing field $\partial_t$,  Corollary \ref{cordiffsplittings} implies that $F$ and $F'$ are almost isometric.  
\end{proof}

However, the applicability of such a correspondence goes further, including results of physical relevance as we will see in the forthcoming section. Before that, let us give the following general result for pre-Randers spaces which will be used later.

\begin{prop}\label{prop:convex}
	Let $(S,F)$ be a pre-Randers metric. Assume that   $d_{s}$ is a distance  and the balls $B_s(x,r)$ are precompact for all $x\in S$ and $r>0$. Then, 
\begin{enumerate}[(i)]	
\item	$(S,F)$ is convex, i.e., any two points on $S$ are connected by a geodesic of $F$.
\item  $F$ is almost isometric to a complete Randers metric. 
\end{enumerate}
\end{prop} 
\begin{proof}
Observe that  all the hypotheses of part GH $\it (ii)$ of Thm. \ref{teo:CSYGH}  hold,  which implies using the same theorem that the associated \som-spacetime is globally hyperbolic   and so causally simple. Then, from part (CS) $\it (iv)$ of Thm. \ref{teo:CSYGH} we deduce that $(S,F)$ is convex.  For part $\it (ii)$, once we know that there exists a standard splitting of the associated \som-spacetime, apply \cite[Theorem 5.10]{CJS11}. 
\end{proof}
Under the conditions of the last proposition it is also possible to obtain multiplicity results of periodic geodesics as in \cite{BiJav11}.
\begin{thm}\label{thm:mult}
	Let $(S,F)$ be a pre-Randers metric with $S$ compact and  assume that $d_{s}$ is a distance. Then there exists a non-trivial periodic geodesic of $(S,F)$. Moreover, if $S$ satisfies one of the following conditions:
	\begin{enumerate}[(i)]
		\item $S$ satisfies the Gromoll-Meyer condition, namely, $\limsup_{k\rightarrow +\infty}b_k(\Lambda S)=+\infty$, where $\Lambda S$ is the loop-space of $S$ and $b_k(\Lambda S)$ is the $k$-th Betti number of $\Lambda S$,
		\item $\dim S\geq 2$ and the fundamental group of $S$ is infinite abelian,
	\end{enumerate}
then $(S,F)$ admits infinitely many geometrically distinct periodic geodesics.
	\end{thm}
\begin{proof}
	As $S$ is compact, Theorem \ref{teo:CSYGH} implies that the associated $\som$-spacetime is globally hyperbolic and it admits a standard stationary splitting by \cite{JS08}. Using Corollary \ref{cordiffsplittings}, it follows that $F$ is almost-isometric to a Randers metric $F'$. As $F$ and $F'$ have the same pregeodesics, the existence of a periodic geodesic follows from \cite{FetLyu51} or \cite{Mer77}, and the multiplicity results, from  Theorems 2.2 and 2.4 of \cite{BiJav11}, in all the cases applied to $F'$.
	\end{proof}

 \begin{rem}
Observe that there is no Hopf-Rinow Theorem for pre-Randers metrics, and probably, Proposition \ref{prop:convex}  is  the closest result that can be obtained. In fact, the example of \S \ref{almoconf}  provides  a compact manifold in which the geodesic connectedness  fails.  In this case,  all the balls with $r>0$ coincide with the whole manifold (which is compact) and the pre-Randers metric is geodesically complete.  Moreover, for  the existence of infinitely many closed geodesics on non-compact Riemannian manifolds see the recent paper \cite{AsMa17}. 
 	\end{rem}

    \subsection{Future horizons and Cut Locus}\label{sec:cutlocus}
Our aim in this section is to present a generalization of the results included in \cite[Section 5.4]{CJS11} involving the so-called \textit{cut locus} of a Randers manifold $(S,R)$ to the context of pre-Randers metrics. In \cite{CJS11}, the authors make use of the relation of Randers metrics and standard stationary spacetimes to characterize the null hypersurface associated with the past of a closed set $\left\{ 0 \right\}\times C\subset \left\{ 0 \right\}\times S$ as the graph of a function involving the distance function associated with the Randers metric.


\medskip

Let us consider $(S,F)$ a pre-Randers manifold. In analogy with \cite[Section 5.4]{CJS11}, let $C\subset S$ be a closed subset of $S$, and define a function $ \rho_{C}:S\setminus C \rightarrow \R$ where $\rho_{C}(x)$ denotes the infimum of $F$-lengths of smooth curves from $x$ to $C$ (for convenience, here we adopt the opposite order to the one used in \cite[\S 5.4]{CJS11}  to define $\rho_C(x)$, where it is defined as infimum of the distances from $C$ to $x$). As happens for Randers metrics, such a function is Lipschitz since $\left| \rho_{C}(x)-\rho_{C}(y) \right|\leq 2 d_{s}(x,y) \leq 2 d_h(x,y)$, where $d_h$ is the distance associated with the Riemannian metric $h$ used in the definition of $F$.  Up to  the last inequalities, take into account that $d_s\geq 0$, or it is identically $-\infty$, and in such a case, $\rho_C=-\infty$, see Proposition \ref{prop:escalera1}. For the second inequality, see Lemma \ref{lem:harris2}. 

The previous function allows us to define the concept of \textit{C-minimizing} segment. However, due to the particularities of pre-Finsler metrics (for instance, a non constant curve could have zero length), we cannot consider the same definition as in \cite[Section 5.4]{CJS11}: we say that a curve $\gamma:I\subset [0,\infty)\rightarrow S$ is C-minimizing if given $a,b\in I$ with $a<b$ it follows that:

\begin{equation}
\rho_{C}(\gamma(a))=\ell_{F}(\gamma)|_{[a,b]} + \rho_{C}(\gamma(b)).\label{cut:eq3}
\end{equation}
Let us prove the following result about the existence of C-minimizing segments (compare with \cite[Proposition 5.11]{CJS11}:

\begin{prop}
	\label{thm:1}
	Every $p\in S\setminus C$ lies on at least one C-minimizing segment {(and so, a geodesic on the pre-Randers manifold)} which arrives to $C$ or it is forward inextendible.
\end{prop}
\begin{proof}
	Observe that it is not possible to extend the proof in \cite[Proposition 9]{CFGH02}, because the spheres of a pre-Randers metric can be empty for every $r>0$ (recall the example in \S \ref{almoconf}). We will proceed using an associated \som-spacetime $(\R\times S,g)$ with the pre-Randers metric (recall the paragraph after \eqref{preRanders}). Observe that given a curve $x:[a,b]\subset\R\rightarrow S$, we can consider a causal curve in $(\R\times S,g)$ defined as $\gamma=(t,x):I\rightarrow \R\times S$ such that $t(s)=\ell_F(x|_{[a,s]})=\int_a^s F(\dot x) d\mu$ (see part $(i)$ of Proposition \ref{teo:caracrel}). Assume that $x_n:[0,1]\rightarrow S$ is a sequence of curves from $p$ to $C$ such that $\lim_n \ell_F(x_n)=\rho_C(p)$. Consider the sequence of causal curves, $\gamma_n=(t_n,x_n):[0,1]\rightarrow \R\times S$ with $t_n(s)=\ell_F(x_n|_{[0,s]})$ and concatenate $\gamma_n$ with the timelike curve $s\ni [1,+\infty)\rightarrow (s,x_n(1))\in\R\times S$. Then choose a future-inextendible causal limit curve $\gamma=(t,x):I=\rightarrow\R\times S$ of the resulting future inextendible sequence of curves with $I$ an interval closed on the left with zero as the left endpoint (for the existence of this limit curve see \cite[Prop. 3.31]{beem1996global}). We claim that $x:I\rightarrow S$ is a C-minimizing segment departing from $p$. If $b\in I$, then there exists a subsequence of the form $\{\gamma_{n_k}(s_k)\}_k$ converging to $\gamma(b)=(t(b),x(b))$. Observe that $(t(b)-\varepsilon,x(b))$, with $\varepsilon>0$, is chronologically related to $\gamma(b)$ as there is a timelike curve, $s\ni [-\varepsilon,0]\rightarrow (s+t(b),x(b))$ joining them. As the chronological future is open, there exists $k_0\in\N$ such that $(t(b)-\varepsilon,x(b))$ is chronologically related to $\gamma_{n_k}(s_k)$ for $k\geq k_0$. Then $(t(b)-\varepsilon,x(b))$ is chronologically  related to $\gamma_{n_k}(1)$ for every $k\geq k_0$ (use \cite[Proposition 10.46]{Oneill83}). This implies, using \eqref{eq:characpasado}, that $d(x(b),x_{n_k}(1))<t_{n_k}(1)-t(b)+\varepsilon$ for $k\geq k_0$ and every $\varepsilon>0$. As, by hypothesis, $x_n(1)\in C$ and $\lim_nt_n(1)=\rho_C(p)$, it follows that 
	\begin{equation}\label{keyineq}
	\rho_C(x(b))\leq \rho_C(p)-t(b). \end{equation}
	On the other hand, as $\gamma$ is a causal curve, we have that $\ell_F(x|_{[0,b]})\leq t(b)$ (using part $(ii)$ of Proposition \ref{teo:caracrel}) and then $\rho_C(p)\leq \ell_F(x|_{[0,b]})+\rho_C(x(b))\leq   t(b)+\rho_C(x(b))$. Putting together last inequality and \eqref{keyineq}, it follows that
\[\rho_C(p)\leq \ell_F(x|_{[0,b]})+\rho_C(x(b))\leq   t(b)+\rho_C(x(b))\leq \rho_C(p).\]
Therefore, $\rho_C(p)= \ell_F(x|_{[0,b]})+\rho_C(x(b))$, and $x$ is a $C$-minimizing segment which is forward  inextendible as $\gamma$ is future inextendible. Finally, in order to show that $\gamma$ is C-minimizing, let us take $a,b\in I$ with $a<b$. Then, from the previous equality used for both, $a$ and $b$,
  \[
\rho_{C}(p)=\ell_{F}(x|_{[0,b]})+\rho_{C}(x(b))=\ell_{F}(x|_{[0,a]})+\rho_{C}(x(a)) ,
  \]
and from the second equality $\rho_{C}(x(a)) =\ell_{F}(x|_{[a,b]})+\rho_{C}(x(b))$.

	\end{proof}
Let us assume that all the C-minimizing segments are defined in their maximal domain. A C-minimizing segment has a \textit{cut point} if its interval of definition is of the form $[a,b]$ or $[a,b)$ with $a>-\infty$, being $p=\gamma(a)$ its cut point. We will define the \textit{cut locus} of $C$ on $S$, denoted by $Cut_{C}$, as the set of all cut points. For any $x\in S\setminus C$, we will also denote by $N_{C}(p)$ the number of C-minimizing segments passing through $p$. It follows from the previous proposition that $N_{C}(p)\geq 1$  and, it is not difficult to see that if $N_{C}(p)\geq 2$ then $p\in Cut_{C}$.

\medskip

The function $\rho_{C}$ and the C-minimizing segments have very natural interpretations on the associated \som-spacetime $(M,g)$, with $M=\R\times S$. For simplicity, let us make the identification $C\equiv \left\{ 0 \right\}\times C\subset \mathbb{R}\times S$ and consider $I^{-}(C)$. Given a point $x\in S\setminus C$,  $(t,x)\in I^{-}(C)$ for all $t<-\rho_{C}(x)$, and so, it follows that $(-\rho_{C}(x),x)\in \overline{I^{-}(C)}\setminus I^{-}(C)$. Moreover, if $\gamma$ is a C-minimizing segment, then $s\rightarrow (-\rho_{C}(\gamma(s)),\gamma(s))$ is a lightlike pre-geodesic on $(M,g)$.


Let us denote by $\mathcal{H}$ the graph of $-\rho_{C}$, i.e.,
\begin{equation}
\label{cut:eq1}
\mathcal{H}=\left\{ (-\rho_{C}(x),x):x\in S\setminus C \right\}.
\end{equation}
Observe that, from construction,  $\mathcal{H}$ is achronal. In fact, if $(-\rho_{C}(x),x)\ll (-\rho_{C}(y),y)$, then by Proposition \ref{teo:caracrel}:

\[
d_{F}(x,y)< -\rho_{C}(y)+\rho_{C}(x)\Rightarrow d_{F}(x,y)+d_{F}(y,C)< d_{F}(x,C),
  \] which contradicts the triangle inequality.

Moreover, defined in this way, $\mathcal{H}$ defines a \textit{future horizon} of the spacetime $(\R\times (S\setminus C),g)$ obtained as an open region of $(M,g)$, i.e., $\mathcal{H}$ is an achronal, closed future null geodesically ruled topological hypersurface (see \cite{CFGH02} and references therein). Recall that \textit{future null geodesically ruled} means that each point $p\in \mathcal{H}$ belongs to a \textit{null generator} of $\mathcal{H}$, i.e., a future inextendible lightlike geodesic. Observe that we have considered the open subset $\R\times (S\setminus C)$, removing $\R\times C$, in order to make the null generators future inextendible. It follows that the number of C-minimizing segments passing through $x\in S\setminus C$ coincides with the number of null generators of $\mathcal H$ passing through $(-\rho_C(x),x)$ and $\mathcal{H}_{\rm end}$ (the set of past endpoints of the null generators of $\mathcal{H}$) coincides with $\{(-\rho_C(x),x):x\in Cut_C\}$. Then, using results of Beem and Krolak  \cite[Theorem 3.5]{BK98} and Chrusciel and Galloway \cite[Prop. 3.4]{CG98}, we get the following.

\smallskip

\begin{thm}
	\label{thm:2} 
	Let $(S,R)$ be a pre-Randers manifold, and $C\subset S$ a closed subset. A point $p\in S\setminus C$ is a differentiable point of the distance function $\rho_{C}$ if, and only if, it is crossed by exactly one minimizing segment, i.e., $N_{C}(p)=1$.
\end{thm}
 Recall that $\rho_C$ is Lipschitz, satisfying in particular that $\left|\rho_C(x) -\rho_C(y)\right|\leq 2d_h(x,y)$ for all $x,y\in S\setminus C$, with $h$ the Riemannian metric in \eqref{preRanders}, and by Rademacher's Theorem is smooth almost everywhere. Moreover, 
 \[d_h(x,y)\leq d_g((-\rho_C(x),x),(-\rho_C(y),y))\leq |\rho_C(x)-\rho_C(y)|+d_h(x,y)\leq 3 d_h(x,y)\] for every $x,y\in S\setminus C$, with $d_g$ the distance associated with the Riemannian metric $g=dt^2+h$ on $\R\times S$. Taking into account these facts and \cite[Theorem 1]{CFGH02},  we obtain the following: 
\begin{cor}
	Given a closed subset $C$ in a pre-Randers manifold $(S,F)$, the subset $Cut_C$ has  $n$-dimensional Hausdorff measure equal to zero, ${\mathfrak h}^n(Cut_C)=0$.
	\end{cor}
%

\subsection{Connectivity of magnetic geodesics}\label{magnetic}
Our aim in this section is to give another application of the correspondence between \som-spacetimes and pre-Randers metrics on a physically relevant situation. Concretely, we will obtain a result on connectivity by magnetic geodesics and existence of periodic  orbits. 

Let $S$ be a  manifold and $\pi:TS\rightarrow S$ the canonical projection from the tangent bundle to $S$. A pair $(\mathfrak{g},\Omega)$, where $\mathfrak{g}$ is a Riemannian metric and $\Omega=-d\omega$ is  an exact  $2$-form on $S$, will be called a {\em magnetic structure} on $S$. We will define a {\em magnetic geodesic} associated with the magnetic structure as a curve $\gamma$ satisfying	
\begin{equation}\label{ELeq}
\frac{D\dot{\gamma}}{dt}=Y_{\gamma}(\dot{\gamma}),
\end{equation}
where $D/dt$ denotes the covariant derivative associated with $\mathfrak{g}$ and $Y_{p}:T_pS\rightarrow T_pS$ is determined by
\[
\Omega_p(u,v)=\mathfrak{g}(Y_p(u),v)
\] 
for all $u,v\in T_pS$. In this case, $\mathfrak{g}(\dot{\gamma},\dot{\gamma})$ is constant and $\frac{1}{2}\mathfrak{g}(\dot{\gamma},\dot{\gamma})$ is called the energy of $\gamma$. Magnetic geodesics model the trajectories of a charged particle of unit mass under the effect of a magnetic field whose Lorentz force is given by $Y$.

Following Ma\~ne's approach \cite{Mane97}, the magnetic geodesics can be presented as critical points of the lagrangian $L$ defined over $S$ and given by
\begin{equation}\label{eq:lag} 
L(v)=\frac{1}{2}\mathfrak{g}(v,v)+\omega(v).
\end{equation}
Let us see the relation between magnetic geodesics and geodesics of pre-Randers metrics. 
\begin{prop}\label{magranders}
Let $(\mathfrak{g}, \Omega)$ be a magnetic structure in a manifold $S$ being $\Omega=-d\omega$ for some one-form $\omega$ on $S$. Then a curve $\gamma:[a,b]\rightarrow S$ is a magnetic geodesic of energy $c>0$  if and only if it is a pre-geodesic of the pre-Randers metric defined by
\begin{equation}\label{Fc}
F_c(v)=\sqrt{\mathfrak{g}(v,v)}+\frac{1}{\sqrt{2c}}\omega(v),
\end{equation}
for every $v\in T S$, parametrized with $\mathfrak{g}(\dot\gamma,\dot\gamma)=2c$.
 
\end{prop}
\begin{proof}
Observe that the Euler-Lagrange equations of the length functional of $F_c$ are given by
$\frac{D}{dt}\left(\frac{\dot{\gamma}}{\sqrt{\mathfrak{g}(\dot{\gamma},\dot{\gamma})}}\right)=\frac{1}{\sqrt{2c}}Y_{\gamma}(\dot{\gamma})$. As the last equation is invariant under reparametrization, the result follows straightforwardly by using \eqref{ELeq}.
\end{proof}
 As  a direct consequence of Prop. \ref{prop:convex}, we obtain the following result.

\begin{cor}\label{cor:magcon}
	Let $S$ be a manifold, $(\mathfrak{g},\Omega)$ an exact magnetic structure on it with $\Omega=- d\omega$, and consider  the pre-Randers metric $F_c$ in $S$ given in \eqref{Fc}.  If there exists a closed one-form $\theta$ such that  
	the symmetrized distance $d_{s}$ of $F=F_c+\theta$  is truly a distance and its corresponding balls $B_s(x,r)$  are precompact for all $x\in S$ and $r>0$, then any two points in $S$ are connected by a magnetic geodesic with energy $c$.
\end{cor}

\begin{proof}
	According to Prop. \ref{prop:convex},  the manifold $(S,F)$ is convex,  so for any two points $p,q\in S$ there exists a  geodesic $\gamma$ of $(S,F)$ joining them. As $F$ and $F_c$ have the same pre-geodesics (recall that $\theta$ is closed) then $\gamma$ is a pre-geodesic of $F_c$, and then reparametrizing $\gamma$ with $\mathfrak{g}(\dot\gamma,\dot\gamma)=2c$, we obtain a magnetic geodesic of  $(\mathfrak{g},\Omega)$ by Proposition \ref{magranders}. 
\end{proof}
Observe that if there exists some point $x\in S$ such that all the loops in $S$ with basepoint $x$ have non-negative $F_c$-length, then the stationary spacetime $(\R\times S,g)$ associated with $F_c$  (see the paragraphs after \eqref{preRanders}  for the definition of the spacetime)  is  chronological,  by part $C2$ of Proposition \ref{prop:escalera1}. 
Let us see the relation of this condition with {\em Ma\~ne critical values}. First let us define the so called {\em strict Ma\~ne critical value}:
\[c_0(L)=\inf\{k\in\R : \sup_{\theta \in\Omega^1( S)} \inf_{(\gamma,T)}\mathbb{S}^\theta_k(\gamma,T)\geq 0 \}
=\inf_{\theta \in\Omega^1( S)}\{\inf\{k\in\R:\inf_{(\gamma,T)}\mathbb{S}^\theta_k(\gamma,T)\geq 0 \}\} ,\]
where $\Omega^1( S)$ is the space of closed one-forms on $S$, \[\mathbb{S}^\theta_k(\gamma,T)=\int_0^T(L(\dot\gamma)+\theta(\dot\gamma)+k)dt\] and $\gamma:[0,T]\rightarrow  S$ is a loop with basepoint on $x$.  Observe that as $\inf_{(\gamma,T)}\mathbb{S}^\theta_k(\gamma,T)$ is non-negative or $-\infty$, it follows that the infimum is attached in the above definition of $c_0(L)$, namely, we can replace $\inf$ by $\min$. Moreover, if we define 
\begin{equation}\label{c(L)}
c(L+\theta)=\min\{k\in\R:\inf_{(\gamma,T)}\mathbb{S}^\theta_k(\gamma,T)\geq 0 \},
\end{equation}
with $L$ as in \eqref{eq:lag}, then one obtains that $c_0=\min_{\theta \in\Omega^1( S)} c(L+\theta)$, a classical expression for the Ma\~ne's critical value \cite[Theorem 1.1]{PaPa97}.  On the other hand, the {\em lowest Ma\~ne critical value} $c_u(L)$ can be computed as $c(\tilde{L})$, using \eqref{c(L)} with $\theta=0$ and being $\tilde L$ the lift of the Lagrangian $L$ in \eqref{eq:lag} to the universal covering $\tilde S$ of $S$ \cite[pag. 484]{PaPa97}. In general, it holds that $c_u(L)\leq c_0(L)$. 
 \begin{prop}\label{prop:mane}
Given a  closed  manifold $S$ and an exact magnetic structure $(\mathfrak{g},\Omega)$, with $\Omega=-d\omega$. Then
\begin{enumerate}[(i)]
\item if $c>c_0(L)$, there exists a closed one-form $\theta$ such that  the loops in $S$ have non-negative $(F_c+\frac{1}{\sqrt{2c}}\theta)$-length,
\item if $c<c_0(L)$, the loops in $S$ are unbounded from below for the $(F_c+\frac{1}{\sqrt{2c}}\theta)$-length.
\end{enumerate}
\end{prop}
\begin{proof}
Observe that given a curve $\gamma:[0,T]\rightarrow  S$, the minimum of $\mathbb{S}^\theta_c(\gamma,T)$ between all its reparametrizations is attained when $\gamma$ is a curve with constant energy $c$, namely, 
$\frac{1}{2}\mathfrak{g}(\dot\gamma,\dot\gamma)=c$. This can be proved, for example, computing the first and second variation of $\int_0^T \frac{1}{2} \mathfrak{g}(\dot\gamma,\dot\gamma) dt$ with a variational vector field proportional to $\dot\gamma$ (the other terms in $\mathbb{S}^\theta_c(\gamma,T)$ are invariant under reparametrizations). The only critical points in this situation are curves with constant velocity and the second variation is always positive.
Assuming that $\frac{1}{2}\mathfrak{g}(\dot\gamma,\dot\gamma)=c$, it follows that   $\mathbb{S}^\theta_c(\gamma,T)=\sqrt{2c}\,\ell_{F_c+\frac{1}{\sqrt{2c}}\theta}(\gamma)$. Then $\inf_{(\gamma,T)}\mathbb{S}^\theta_{c}(\gamma,T)=-\infty$ if there is some loop with negative $F_c+\frac{1}{\sqrt{2c}}\theta$-length and $\inf_{(\gamma,T)}\mathbb{S}^\theta_{c}(\gamma,T)\geq 0$ otherwise. This easily implies $(i)$ and $(ii)$.
\end{proof}
Recall that the Hamiltonian associated with the Lagrangian $L$ in \eqref{eq:lag} is given by
\begin{equation}\label{HdeL}
H(p)=\frac{1}{2} |p-\omega|^2_{\mathfrak{g}}.
\end{equation}
Next, we recall a result from  \cite[Corollary 2]{CIPP98} (see also \cite[Theorem 4.1]{Abb13}), which establishes a relation between magnetic geodesics of a certain energy and Finsler metrics.
\begin{thm}\label{thm:conj}
Let $S$ be a closed manifold and $(\mathfrak{g},\Omega)$ a magnetic structure on it. Then the magnetic geodesics with energy $k>c_0(L)$ are conjugated to  the geodesics of a Randers metric of the form $\tilde F_k(v)=\sqrt{2k\mathfrak{g}(v,v)}+ \omega(v)-\alpha(v)$, where $\alpha$ is a closed one-form which satisfies $H(\alpha)<k$.
\end{thm}
\begin{proof}
Observe that following the proof of \cite[Theorem 4.1]{Abb13}, the Finsler metric which is conjugated to the level $k>c_0(L)$ is obtained as follows. Let us define the Hamiltonian $K(p)=H(p+\alpha)$. We have that $K^{-1}(k)$ is the boundary of a uniformly convex bounded open set which contains the zero section of $T^*M$. Then we can define a two-homogeneous function $F:T{^*}M\rightarrow [0,+\infty)$  determined by the condition $F^{-1}(1)=K^{-1}(k)$. The Finsler metric is the Legendre dual of $F$. 
 By definition, if $p\in K^{-1}(k)$ and $\pi^*:T^*M\rightarrow M$ is the natural projection, then
\begin{equation}
k=H(p+\alpha)=\frac{1}{2} |p+\alpha-\omega|^2_{\mathfrak{g}}=\frac{1}{2} \max_{v\in T_{\pi^*(p)}M}|p(v)+\alpha(v)-\omega(v)|^2/g(v,v).
\end{equation}
In particular,  $p\in K^{-1}(k)$ if and only if $p(v)+\alpha(v)-\omega(v)\leq \sqrt{2k\mathfrak{g}(v,v)}$ for all $v\in T_{\pi^*(p)}M$ and the equality is reached for some $v_0\in T_{\pi^*(p)}M$. The last condition is equivalent to
\[1=\max_{v\in T_{\pi^*(p)}M} \frac{p(v)}{\tilde F_k(v)}=\tilde H_k(p),\]
where $\tilde H_k$ is the Legendre dual of $\tilde F_k$. This concludes that $F=\tilde H_k$, and then the Finsler metric conjugated to the level $k$ is $\tilde F_k$, as required. 
\end{proof}
Observe that the pre-Randers metrics in Proposition \ref{prop:mane} and Theorem \ref{thm:conj} are related. Indeed, $\tilde F_k=\sqrt{2k}F_k$, when $\alpha=-\theta$, and both metrics have the same geodesics.

Finally, we can also obtain an application for stationary spacetimes using some well-known results for periodic magnetic geodesics. Recall that in an $\som$-spacetime $(\R\times S,g)$, we say that a curve is $t$-periodic if its $t$-component is closed.
 Let us consider a magnetic structure associated with a the stationary-complete manifold given in \eqref{def:metrica}. In this case, we will consider the Lagrangian
\begin{equation}\label{Lstat}
L(v)=\frac{1}{2\beta^2}\omega(v)^2+\frac{1}{2\beta}g_0(v,v)+\frac{1}{\beta}\omega(v).
\end{equation}
\begin{cor}
	Let $(\R\times S,g)$ be the stationary-complete spacetime defined in \eqref{def:metrica}, with $S$ compact. If the lowest critical Ma\~ne value satisfies $c_u(L)< \frac{1}{2}$, being $L$ the Lagrangian defined in \eqref{Lstat}, then there exists a $t$-periodic lightlike geodesic of $(\R\times S,g)$.
	\end{cor}
\begin{proof}
	The existence of a $t$-periodic lightlike geodesic of $(\R\times S,g)$ is equivalent to the existence of a periodic geodesic of $(S,F)$, with $F$ the pre-Finsler metric defined in \eqref{eq:wsfinsler} (see Proposition \ref{prop:Fermatmetric} and also \cite[Theorem 7]{Jav13}). Moreover, by Proposition \ref{magranders}, pre-geodesics of $F$ are reparametrizations of magnetic geodesics of the magnetic structure $(\frac{1}{2\beta^2}\omega^2+\frac{1}{2\beta}g_0,-d(\frac{1}{\beta}\omega))$ on $S$ with energy $c$ satisfying $1=\frac{1}{\sqrt{2c}}$, namely, $c=1/2$. If $c_u(L)<1/2$, then by the main theorem of \cite{Abb13}, there exists a periodic magnetic geodesic of energy $c=1/2$, as required.
	\end{proof}

\begin{cor}
Let $\som(\R\times S)$ be the space of stationary spacetimes admitting a non-standard splitting with $S$ compact. Then almost all of them admit periodic lightlike geodesics.
\end{cor}
\begin{proof}
This is a straightforward consequence of the existence of magnetic geodesics for almost all the possible values of the energy $c$  (see the main theorem of \cite[page 399]{Abb13}).
\end{proof}
\appendix
\section{The cocycle approach to the causal ladder }

Some of the steps of the causal ladder for \som-spacetimes were studied first by Harris in \cite{Harris15}. The approach is different from the one presented here, as it is presented in terms of algebraic structures. However it is possible to obtain remarkable relations between both approaches.
Along this appendix, and after giving a brief review on the approach in \cite{Harris15}, we will show how the algebraic structures of such an approach are related with pre-Finsler metrics.

\subsection{Algebraic approach for global causality in \som-spacetimes}

As we have mentioned on \S \ref{s:stat}, the space of Killing orbits is a Hausdorff manifold. Let $Q$ denote such a space and define the projection $\pi:M\rightarrow Q$ as the quotient by the $\R$-action on $M$ defined by $t\cdot x=\gamma_{x}(t)$, where $\gamma_{x}$ is the integral curve of $K$ with $\gamma_{x}(0)=x$.

Then, it is possible to define the {\em Killing time function} as any function $t:M\rightarrow \R$ satisfying $dt(K) =1$. Observe that such a function defines naturally a representation of $Q$ as a slice on $M$ (with the same role as $S$ in our approach) making the identification $Q\equiv t^{-1}(0)$. 

It is proved then that, associated with this Killing time function $t$ (or, equivalently, to the slice $z_{t}$) there exists a one form $\tilde{\omega}$ (called the {\em Killing  drift-form}), such that

\begin{equation}
  \label{eq:1}
  g=-(\Omega\circ \pi)\left(dt+\pi^*\tilde{\omega}\right)^2 + \pi^* h,
\end{equation}
where $\Omega:Q\rightarrow \R^+$ is given by the length-squared of $K$ and $h$ is a Riemannian metric defined on $Q$ (thanks to the Killing character of $K$). The relation between the coefficients of \eqref{def:metrica} and \eqref{eq:1} is given by  

%

\[
\beta=\Omega,\qquad \omega=-\Omega\tilde{\omega}, \qquad  g_{0}=-\Omega \tilde{\omega}^2 + h\qquad  \hbox{and}\qquad S=z_t,
  \]
  where we have omitted any reference to $\pi$ for simplicity.
  Moreover, and recalling the definition of $F$ (see (\ref{eq:wsfinsler})), it follows that in this case:

  \begin{equation}\label{fermatpre}
F(v)=-\tilde{\omega}(v)+\sqrt{\frac{h(v,v)}{\Omega}}.
    \end{equation}

    \medskip
    
 Now, we will introduce the basic tools needed for the characterization given in \cite{Harris15}. We will stick in the essential tools required for understanding the relation between both papers, referring Harris' paper for those readers interested on further details.

 For a given curve $c:[a,b]\rightarrow S$, we will denote by $L(c)$ the length of $c$ computed with the Riemannian metric $\tilde{h}:=h/\Omega$. We will also define the {\em efficiency} of the curve $c$ with respect to a $1$-form $\theta$, denoted by $eff_{\theta}(c)$, by

 \[
eff_{\theta}(c)=\frac{\int_{c}\theta}{L(c)}
\]
and

\[
L_{\theta}(c)=L(c)-\int_{c}\theta.\footnote{In fact, in Harris' paper, he defines $d_{F}$ as $d_{\tilde\omega}$ and recall some properties of it. For instance, our characterization of chronologically vicious is implicit in \cite[Proposition 2.9]{Harris15}}
  \]
 Observe that $L_{\tilde\omega}(c)=\ell_F(c)$, namely, the length with the Fermat metric given in \eqref{fermatpre}. Finally, define the {\em weight} of $\theta$, denoted by $wt(\theta)$, as

   \[
wt(\theta)= {\rm sup}_{loops\, c}(eff_{\theta}(c)).
     \]

     \begin{rem}
       Let us remark that here we are making a {\em huge} simplification of the approach, as different technicalities should be considered. For instance, here the weight is defined for $1$-forms, while in the original approach is done for {\em cocycles}. Moreover, the previous supreme is taken over the set of loops (closed curves) with the same basepoint, so it is necessary to recall that the weight is independent of the base point considered. Again, we refer to the original work for a detailed presentation of all these technicalities. 
     \end{rem}

     In order to present the main result in Harris' paper regarding the causal ladder, we need to introduce the following concepts:

     \begin{defi}
       Let $\pi:M\rightarrow S$ be the projection from $M$ to $S$. We will say that $M$ is {\it causally bounded} if for all $p,p'\in M$, $\pi(I^+(p)\cap I^-(p'))$ is bounded for the metric $\tilde{h}$. $M$ is {\it spatially complete} if $S$ is complete for $\tilde{h}$.
     \end{defi}

     These definitions are in the {\rm core} of the global hyperbolicity of $(M,g)$. In fact, it follows (see \cite[Proposition 6.3]{Harris15})

     \begin{prop}\label{prop:harris1}
       $M$ is globally hyperbolic if and only if $M$ is
       \begin{enumerate}
       \item future-distinguishing,
       \item causally bounded, and
       \item spatially complete.
       \end{enumerate}
     \end{prop}

     \smallskip 

     The main result in \cite{Harris15} then reads:

     \begin{thm}\label{HarrisThm}
       Let $\pi:M\rightarrow S$ be a stationary-complete spacetime satisfying the observer-manifold condition. There are only these mutually exclusive possibilities:

       \begin{enumerate}
       \item If $wt(\tilde\omega)>1$, then $M$ is chronologically vicious.

       \item If

         \begin{itemize}
         \item $wt(\tilde\omega)=1$ and
         \item there is a loop $c$ in $S$ with $L(c)=\int_c \tilde\omega$,
         \end{itemize}
         then $M$ is chronological but not causal.

       \item If

         \begin{itemize}
         \item $wt(\tilde\omega)=1$,
         \item for all loop $c$ in $S$ with $L(c)>\int_c \tilde\omega$ and
         \item there is a sequence of base-pointed loops $\{c_n\}_n$ in $S$ with
           \[
eff_{\tilde\omega}(c_n)\rightarrow 1 \quad \hbox{and} \quad L_{\tilde\omega}(c_n)\rightarrow 0;
             \]
         \end{itemize}
         then $M$ is causal but not future- or past-distinguishing (in particular, not strongly causal).

\item If
\begin{itemize}
         \item $wt(\tilde\omega)=1$,
         \item for all loop $c$ in $S$ with $L(c)>\int_c \tilde\omega$,
         \item for every sequence of base-pointed loops $\{c_n\}_n$ in $S$ with $eff_{\tilde\omega}(c_n)\rightarrow 1$, it follows that $\{L_{\tilde\omega}(c_n)\}_n$ is bounded away from $0$.
         \item there is such a sequence $\{c_n\}_n$ with $\{L_{\tilde\omega}(c_n)\}$ bounded above;    
         \end{itemize}
         then $M$ is strongly causal  (and causally continuous) but not spatially complete or not causally bounded (in particular, not globally hyperbolic).

\item If
\begin{itemize}
         \item $wt(\tilde\omega)=1$,
         \item for all loop $c$ in $S$ with $L(c)>\int_c \tilde\omega$,
         \item for every sequence of base-pointed loops $\{c_n\}_n$ in $S$ with $eff_{\tilde\omega}(c_n)\rightarrow 1$, it follows that $L_{\tilde\omega}(c_n)\rightarrow \infty$,    
         \end{itemize}
         then $M$ is strongly causal  (and causally continuous) and causally bounded (so globally hyperbolic if, and only if, it is spatially complete).  

\item If $wt(\tilde\omega)<1$, then $M$ is strongly causal and causally bounded.
         
       \end{enumerate}
     \end{thm}
    
\subsection{Comparing both approaches}
The natural question at this point is clear, how the previous concepts are related with the pre-Finsler metric given in \eqref{fermatpre} (and related objects, like the symmetrized distance). Of course, we can infer such relations by means of the several characterizations of the causal ladder obtained here, however a direct approach (that is, comparing directly with the pre-Finsler metric) will be more illuminating.

Let us start by the following two propositions.

     \begin{prop}
       $wt(\tilde{\omega})\leq 1$ if, and only if, $d_{F}(x,x)=0$ for some (and then, all) $x\in S$.
     \end{prop}
     \begin{proof}
           The equivalence follows by simply observing that, for any loop $c$ with basepoint $x_0$
\[
{\ell}_{F}(c)=L(c)-\int_{c}\tilde{\omega}=L(c)\left( 1-eff_{\tilde\omega}(c) \right)
\]
Hence ${\ell}_{F}(c)<0$ if and only if $eff_{\tilde\omega}(c)> 1$. 
\end{proof}

\begin{prop}\label{prop:prop1}
  If $d_s(x_0,x_1)=0$ for some $x_0\neq x_1$, then there exists a family of cycles $\{c_n\}$ with basepoint in $x_0$, not contractible\footnote{Here we say that the sequence of curves $\left\{ c_{n} \right\}_{n}$ is not contractible if there exists a sequence $\{t_n\}_{n}$ so $c_n(t_n)\not\rightarrow x_0$.}  and  such that $L_{\tilde\omega}(c_n)\rightarrow 0$. The converse is also true if we assume additionally that $eff_{\tilde\omega}(c_n)\rightarrow 1$.
\end{prop}
Before the proof of this proposition, we will need to recall the following result (which proof can be deduced from \cite[Proposition 2.2]{Harris15})

\begin{lemma}
 Let $\{c_n\}_n$ be a family of loops with basepoint $x_0$ and  such that $eff_{\tilde\omega}(c_n)\rightarrow 1$. Then, there exists a neighborhood $U$ (homeomorphic to the Euclidean ball) with $x_0\in U$ such that for all $n$ big enough, the curve $c_n$ leaves $U$.
\end{lemma}

\begin{proof}[Proof of Proposition \ref{prop:prop1}]
    Assume that $d_{s}(x_0,x_1)=0$ for some $x_0,x_1\in S$ with $x_0\neq x_1$. From the definition of $d_s$, we deduce that there exists a family of loops $\{c_n\}_n$ with basepoint $x_0$, passing through $x_1$ and with $\ell_{F}(c_n)\rightarrow 0$. Then the result follows recalling that $\ell_{F}(c_n)=L_{\tilde\omega}(c_n)$.

    For the other implication, assume that we have a family of loops $\{c_n\}_n$ with basepoint $x_0$ and satisfying both, that $L_{\tilde\omega}(c_n)\rightarrow 0$ and $eff_{\tilde\omega}(c_n)\rightarrow 1$. Let $U$ be the neighborhood given in the previous lemma, and denote by $y_n$ the first contact point of $c_n$ with $\partial U$ (which is assumed to be compact). From compactness, and up to a subsequence, we can assume that $y_n\rightarrow x_1$ for some $x_1\in \partial U$ (and so, with $x_0\neq x_1$). Now, from the continuity of $d_s$ (which is derived from the continuity of $d_{F}$, see  Prop. \ref{prop:continuidad}) we have:

    \[
d_s(x_0,x_1)={\rm lim}_{n} d_{s}(x_0,y_n)\leq {\rm lim}_n \ell_{F}(c_n)={\rm lim}_n L_{\tilde\omega}(c_n)=0
      \]
\end{proof}

Both previous propositions allow us to understand completely how  parts $\it (1)$, $\it (2)$, $\it (3)$ and even $\it (4)$  of Theorem \ref{HarrisThm}  are related with Props. \ref{prop:escalera1} and \ref{prop:disting}. It remains then to understand how the spatial completeness and the causal boundedness are related with $d_{F}$ and $d_s$. In order to achieve this, let us recall the following property:

\begin{lemma}\label{lem:Harris1}
  Let $\pi:M\rightarrow S$ be the projection to $S$. Take two points  $(\Omega_0,x_0),(\Omega_1,x_0)\in M$ on the spacetime. Then:

  \[
\pi(I^+((\Omega_0,x_0))\cap I^-((\Omega_1,x_0)))=B_s(x_0,\frac{\Omega_1-\Omega_0}{2})
    \]
\end{lemma}
\begin{proof}
  Let us start by showing that the  subset to the left is included in the one to the right. For this, take $(t,x)\in I^+((\Omega_0,x_0))\cap I^-((\Omega_1,x_0))$. From Prop. \ref{teo:caracrel}, it follows that

  \begin{equation}\label{eq2} 
    \left.\begin{array}{l}
      d_{F}(x,x_0)<\Omega_1-t\\ d_{F}(x_0,x)<t-\Omega_0
      \end{array}\right\}\Rightarrow d_{s}(x,x_0)<\frac{\Omega_1-\Omega_0}{2}.
  \end{equation}

  Therefore,  $B_s(x_0,(\Omega_1-\Omega_0)/2)\supset \pi(I^+((\Omega_0,x_0))\cap I^-((\Omega_1,x_0)))$. For the other inclusion,    take $x\in B_s(x_0,(\Omega_1-\Omega_0)/2)$ and observe that, then, there exists $t\in \R$ such that the inequalities on (\ref{eq2}) follows. Then, $(t,x)\in I^+((\Omega_0,x_0))\cap I^-((\Omega_1,x_0))$, and so, $B_s(x_0,(\Omega_1-\Omega_0)/2)\subset \pi(I^+((\Omega_0,x_0))\cap I^-((\Omega_1,x_0)))$ and the result follows.
\end{proof}

Previous characterization is key in order to relate causal boundedness and $d_s$. In fact, let us remark that causal boundedness can be defined by using only points $p,p'\in M$ with $\pi(p)=\pi(p')$. For a given $x_0\in S$ and $(\Omega',x_1)$ we can always find $\Omega_1$ such that $(\Omega',x_1)\ll (\Omega_1,x_0)$. Hence we have that

\[
I^+((\Omega_0,x_0))\cap I^-((\Omega',x_1))\subset I^+((\Omega_0,x_0))\cap I^-((\Omega_1,x_0)).
  \]

  In conclusion, it follows that:

  \begin{prop}
   Let $M$ be an \som-spacetime. Then $M$ is causally bounded if, and only if, the balls of  the symmetrized distance are bounded with the Riemannian metric $\tilde{h}$. 
 \end{prop}

 It is remarkable, and particularly useful {\em a posteriori}, to recall that the balls of  the Riemannian metric $\tilde{h}$ are always bounded for the symmetrized distance. In fact,

 \begin{lemma}\label{lem:harris2}
   For any $x_0\in S$ and $r>0$ it follows that:

   \[
B_{\tilde{h}}(x_0,r)\subset B_{s}(x_0,r)
     \]
 \end{lemma}

 \begin{proof}
   Consider $x_0,x_1\in S$ two arbitrary points. Let $c:[0,1]\rightarrow S$ be any curve joining $x_0$ to $x_1$ and define $(-c):[0,1]\rightarrow S$ as $(-c)(t)=c(1-t)$. It follows directly from definition that the concatenation $\gamma$ of both curves $c$ and $(-c)$ gives a loop of base $x_0$, passing through $x_1$ and such that:

   \[
\ell_{F}(\gamma)=2 L(c)
\]
Hence $d_{s}(x_0,x_1)\leq \frac{\ell_{F}(\gamma)}{2}=L(c)$. As $c$ is an arbitrary curve from $x_0$ to $x_1$, it follows then that $d_{s}(x_0,x_1)\leq d_{\tilde{h}}(x_0,x_1)$, therefore $B_{\tilde{h}}(x_0,r)\subset B_{s}(x_0,r)$ as desired.

 \end{proof}

 In conclusion, and looking to Proposition \ref{prop:harris1}, we have that: (a) the condition of future-distinguishing parallels the condition of that $d_s$ is  a (truly) distance (recall Prop. \ref{prop:disting}) and (b) both conditions causally bounded and spatially complete imply the pre-compactness of the symmetrized balls. In fact, we can prove:

 \begin{prop}
   Let $M$ be an \som-spacetime.  $M$ is both causally bounded and spatially complete if and only if the symmetrized balls are pre-compact. 
 \end{prop}
 \begin{proof}
   For the right implication, take the closure of any symmetrized ball $\overline{B}_s(x_0,r)$. From the causal boundedness  and Lemma \ref{lem:Harris1},  there exists $R>0$ such that $\overline{B}_s(x_0,r)\subset \overline{B}_{\tilde{h}}(x_0,R)$, being the latter compact from the spatial completeness. Then, as $\overline{B}_s(x_0,r)$ is closed, it is also compact.

   For the left implication, let us begin by proving that $M$ is spatially complete. For this, recall that any Cauchy sequence $\{x_n\}$ for $\tilde{h}$ will be contained in a ball $B_{\tilde{h}}(x_0,r)$ for some $x_0\in S$ and $r>0$ big enough. From Lemma~\ref{lem:harris2}, it follows that the previous ball is contained in a symmetrized ball, which is pre-compact by hypothesis. Then, the sequence $\{x_n\}$ is convergent. Finally, for the causal boundedness, recall that, by hypothesis, $\overline{B}_s(x_0,r)$ is compact for  every  $x_0$ and $r>0$, and so, bounded for any Riemannian metric considered (in particular for $\tilde{h}$). Hence, the causal boundedness follows from Lemma~\ref{lem:Harris1}. 
 \end{proof}
\section*{Acknowledgments}
 The authors warmly acknowledge Professor Marco Mazzucchelli (\'Ecole Normale Sup\'erieure de Lyon) for very helpful discussions and  comments on a preliminary version of the article. 

 \end{document}